\documentclass[a4paper,11pt]{amsart}
\usepackage{amscd,amssymb,amsmath,enumerate,color}
\usepackage[all]{xy}

\newcommand{\Mod}{\operatorname{Mod}\nolimits}
\newcommand{\Gr}{\operatorname{Gr}\nolimits}

\newcommand{\Hom}{\operatorname{Hom}\nolimits}

\renewcommand{\Im}{\operatorname{Im}\nolimits}

\newcommand{\rrad}{\mathfrak{r}}

\newcommand{\Ext}{\operatorname{Ext}\nolimits}
\newcommand{\gldim}{\operatorname{gldim}\nolimits}

\renewcommand{\L}{\Lambda}
\renewcommand{\l}{\lambda}
\newcommand{\cQ}{{\mathcal Q}}

\renewcommand{\b}{\bar}
\newcommand{\p}{\oplus}
\renewcommand{\a}{\alpha}

\renewcommand{\b}{\beta}
\newcommand{\g}{\gamma}
\newcommand{\R}{{\mathcal R}}
\newcommand{\Tip}{\operatorname{Tip}\nolimits}
\newcommand{\CTip}{\operatorname{CTip}\nolimits}
\newcommand{\Supp}{\operatorname{Supp}\nolimits}

\newtheorem{lem}{Lemma}[section]
\newtheorem{prop}[lem]{Proposition}
\newtheorem{cor}[lem]{Corollary}
\newtheorem{thm}[lem]{Theorem}
\theoremstyle{definition}
\newtheorem{defin}[lem]{Definition}

\newtheorem*{remark*}{Remark}
\newtheorem{example}[lem]{Example}

\begin{document}

\topmargin 0cm
\oddsidemargin 0.5cm
\evensidemargin 0.5cm
\baselineskip=15pt

\title[The Ext algebra and a new generalisation of $D$-Koszul algebras]
{The Ext algebra and a new generalisation of $D$-Koszul algebras}

\author[Leader]{Joanne Leader}
\address{Joanne Leader\\
Department of Mathematics \\
University of Leicester \\
University Road \\
Leicester LE1 7RH \\
United Kingdom}
\email{jml47@student.le.ac.uk}
\author[Snashall]{Nicole Snashall}
\address{Nicole Snashall\\
Department of Mathematics \\
University of Leicester \\
University Road \\
Leicester LE1 7RH \\
United Kingdom}
\email{njs5@le.ac.uk}

\thanks{This work contributed to the first author's PhD thesis at the University of Leicester, which was supported by EPSRC}


\begin{abstract}
We generalise Koszul and $D$-Koszul algebras by introducing a class of graded algebras called $(D,A)$-stacked algebras. We give a characterisation of $(D,A)$-stacked algebras and show that their Ext algebra is finitely generated as an algebra in degrees $0, 1, 2$ and $3$. In the monomial case, we give an explicit description of the Ext algebra by quiver and relations, and show that the ideal of relations has a quadratic Gr\"obner basis; this enables us to give a regrading of the Ext algebra under which the regraded Ext algebra is a Koszul algebra.
\end{abstract}

\maketitle

\section*{Introduction}
In this paper we introduce a class of graded algebras called $(D,A)$-stacked algebras which provide a natural generalisation of both Koszul algebras and $D$-Koszul algebras (\cite{bib:ber}) as well as the $(D,A)$-stacked monomial algebras of Green and Snashall (\cite{bib:GS2}, \cite{bib:GS3}).
There are many generalisations of Koszul algebras, and, in introducing $(D,A)$-stacked
algebras, we retain the essential property that each projective in a minimal graded projective resolution of $\L_0$ is generated in a single degree, and show that the Ext algebra is finitely generated. Moreover, it is well-known that the Ext algebra of a Koszul algebra is again a Koszul algebra, and we show that this extends to $(D,A)$-stacked monomial algebras, in that there is a regrading of the Ext algebra so that the regraded Ext algebra is a Koszul algebra.

Let $K$ be a field and let $\Lambda = K\cQ/I$ be a finite-dimensional algebra where $I$ is an admissible ideal of $K\cQ$. If the ideal $I$ is generated by length homogeneous elements then there is an induced length grading on $\Lambda$ so that $\Lambda = \Lambda_0 \oplus \Lambda_1 \oplus \Lambda_2 \oplus \cdots$. We suppose throughout this paper that $\Lambda = K\cQ/I$ is a graded finite-dimensional algebra with the length grading.
Let $\rrad$ denote the graded Jacobson radical of $\Lambda$, so $\rrad = \Lambda_1 \oplus \Lambda_2
\oplus \cdots$ and $\Lambda_0 \cong \Lambda/\rrad$ is a finite product of copies of $K$.
In this setting, graded $\Lambda$-modules have minimal graded projective resolutions.
So, viewing $\Lambda_0$ as a graded right $\Lambda$-module, there is a minimal graded projective
resolution $(P^n, d^n)$ of $\Lambda_0$. Recall that this resolution is minimal if
$\Im d^n \subseteq P^{n-1}\rrad$ for all $n \geqslant 1$.
The Ext algebra of $\Lambda$, $E(\L)$, is defined to be $E(\Lambda) = \bigoplus_{n\geqslant 0}\Ext^n_\Lambda(\Lambda_0,\Lambda_0)$ where multiplication is given by the Yoneda product.

The structure of the Ext algebra of a graded finitely-generated algebra is not in general well understood,
and it is often difficult to check whether the Ext algebra is itself finitely generated as an algebra.
It is well-known that the Ext algebra of a Koszul algebra is finitely generated in degrees 0 and 1
(\cite{bib:BGW, bib:GMV}), and this provides the starting point for the class of algebras we study in this paper.
Recall that a graded algebra $\L = \L_0 \p \L_1 \p \L_2 \p \cdots $
is said to be Koszul if $\L_0$ has a linear resolution, that is, if the $n$th projective module
$P^n$ in a minimal graded projective resolution $(P^n, d^n)$ of $\Lambda_0$ is generated in degree $n$.
Moreover, if $K\cQ / I$ is Koszul then $I$ is a quadratic ideal, that is, $I$ is generated by linear
combinations of paths of length $2$.

The paper is structured as follows. In Section~\ref{DA}, we define the class of $(D,A)$-stacked algebras and discuss their relationship to Koszul algebras, the $D$-Koszul algebras introduced by Berger (\cite{bib:ber}), the $(D,A)$-stacked monomial algebras of Green and Snashall (\cite{bib:GS2}) and the $\delta$-Koszul algebras of Green and Marcos (\cite{bib:GM}).
Our first result is Theorem~\ref{genin}, where we show that the Ext algebra of a $(D,A)$-stacked algebra
is finitely generated as an algebra in degrees $0, 1, 2$ and $3$. We then study the properties of the
Ext algebra of a $(D,A)$-stacked algebra and, in Theorem~\ref{charac}, provide a
characterisation of $(D,A)$-stacked algebras.
We give a regrading of the Ext algebra in Section~\ref{regrad} and recall some necessary Gr\"obner basis theory in Section~\ref{sec:GB}.
Section~\ref{sec:extmonomial} focuses on the monomial case, where we give an explicit description of the Ext algebra $E(\Lambda)$ by quiver and relations for a $(D,A)$-stacked monomial algebra $\Lambda$, so that $E(\Lambda)\cong K\Delta/I_\Delta$. In particular, we show in Proposition~\ref{prop:monredGB} that $I_\Delta$ has a quadratic reduced Gr\"obner basis in the case where $A > 1$ and $D \neq 2A$.
We then prove in Theorem~\ref{thm:main_monomial_thm_2} that the regraded Ext algebra of a $(D,A)$-stacked monomial algebra is a Koszul algebra, providing $D \neq 2A$ when $A > 1$. This generalises the previous results for Koszul and $D$-Koszul algebras (\cite[Theorem 7.1]{bib:GMMZ}), and provides further justification for these algebras being considered a natural generalisation of Koszul algebras.
The final section returns to the general case, where we give a non-monomial example of a $(D,A)$-stacked algebra where the regraded Ext algebra is also Koszul. We conclude the paper with some open questions.

\bigskip

Throughout this paper, $\Lambda = K\cQ/I$ is a graded finite-dimensional algebra where $I$ is length homogeneous. An arrow $\alpha$ starts at $s(\alpha)$ and ends at $t(\alpha)$; arrows in a path are read from left to right. A path $p = \alpha_1\alpha_2 \cdots \alpha_n$, where  $\alpha_1, \alpha_2 , \dots , \alpha_n$ are arrows, is of length $n$ with $s(p) = s(\alpha_{1})$ and $t(p) = t(\alpha_{n})$. If the ideal $I$ is generated by paths in $K\cQ$ then $K\cQ / I$ is a monomial algebra. An element $x$ in $K\cQ$ is uniform if there exist vertices $e_i, e_j \in \cQ_0$ such that $x = e_i x = x e_j$.

\section{$(D,A)$-Stacked Algebras}\label{DA}

In this section we introduce $(D,A)$-stacked algebras where $D \geqslant 2, A \geqslant 1$. The parameters $D$ and $A$ relate to the degree in which the projective modules $P^2$ and $P^3$ are generated, in a minimal projective resolution $(P^n, d^n)$ of $\Lambda / \rrad$. We then prove that the Ext algebra of a $(D,A)$-stacked algebra is finitely generated. We denote the global dimension of $\Lambda$ by $\gldim \L$.

\begin{defin}\label{d,a}
Let $\Lambda = K \cQ / I$ be a finite-dimensional algebra. Then $\Lambda$ is a $(D,A)$-stacked algebra if there is some $D \geqslant 2, A \geqslant 1$ such that, for all $0 \leqslant n \leqslant \gldim \L$, the projective module $P^n$ in a minimal projective resolution of $\Lambda / \rrad$ is generated in degree $\delta(n)$, where
$$\delta(n) = \begin{cases}
0 & \mbox{if $n = 0$} \\
1 & \mbox{if $n = 1$} \\
\frac{n}{2} D & \mbox{if $n$ even, $n \geqslant 2$ }\\
\frac{(n-1)}{2} D + A & \mbox{if $n$ odd, $n \geqslant 3$.}
\end{cases} $$
\end{defin}

\begin{remark*}
It is clear from the definition that if $\Lambda = K \cQ / I$ is a $(D,A)$-stacked algebra then the projective module $P^2$ is generated in degree $D$. Thus the ideal $I$ is length homogeneous and is generated by linear combinations of paths of length $D$, where $D \geqslant 2$. Hence $\L$ is necessarily a graded algebra with the length grading, and $\Lambda/\rrad \cong \Lambda_0$.
\end{remark*}

The $(D,A)$-stacked algebras with $A = 1$ are precisely the algebras where, for $0 \leqslant n \leqslant \gldim \L$, the projective module $P^n$ is generated in degree
$$\begin{cases}
\frac{n}{2} D & \mbox{if $n$ even, $n \geqslant 0$}\\
\frac{(n-1)}{2} D + 1 & \mbox{if $n$ odd, $n \geqslant 1$.}
\end{cases} $$
It follows that the $(D,1)$-stacked algebras are precisely the finite-dimensional $D$-Koszul algebras of Berger (\cite{bib:ber}). Berger introduced $D$-Koszul algebras in order to include some cubic Artin-Schelter regular algebras and anti-symmetrizer algebras. It was shown by Green, Marcos, Mart\'inez-Villa and Zhang in \cite[Theorem 4.1]{bib:GMMZ} that the Ext algebra of a $D$-Koszul algebra is finitely generated in degrees $0, 1$ and $2$. Koszul algebras are $D$-Koszul algebras with $D = 2$ and so the finite-dimensional Koszul algebras are precisely the $(D,A)$-stacked algebras where $D = 2, A = 1$. In this case, it is well-known that their Ext algebra is generated in degrees 0 and 1. We note also that the $(D,A)$-stacked algebras share the property of Koszul and $D$-Koszul algebras that each projective module in a minimal projective resolution of $\Lambda_0$ is generated in a single degree.

The $(D,A)$-stacked algebras clearly contain the $(D,A)$-stacked monomial algebras of Green and Snashall
\cite[Definition 3.1]{bib:GS2} (and see \cite{bib:GS3}).
For monomial algebras of infinite global dimension, the $(D, A)$-stacked monomial algebras are precisely
the monomial algebras for which every projective module $P^n$ in a minimal graded projective resolution
$(P^n, d^n)$ of $\Lambda_0$ is generated in a single degree and for which the Ext algebra $E(\L)$ is finitely
generated. It was shown in \cite[Theorem 3.6]{bib:GS3} that the Ext algebra of a
$(D,A)$-stacked monomial algebra is finitely generated in degrees $0, 1, 2$ and $3$.

In \cite{bib:GM}, Green and Marcos defined a graded algebra $\Lambda = \L_0 \p \L_1 \p  \cdots$
to be $\delta$-resolution determined if there is a function $\delta:\mathbb{N} \rightarrow \mathbb{N}$ such that, for all $0 \leqslant n \leqslant \gldim \L$, the $n$th projective $P^n$ in a minimal graded projective resolution $(P^n, d^n)$ of $\Lambda_0$ over $\Lambda$ is generated in degree $\delta(n)$. Additionally, a $\delta$-resolution determined algebra $\Lambda$ is said to be $\delta$-Koszul if $E(\Lambda)$ is finitely generated as an algebra.
It is clear that a $(D,A)$-stacked algebra is $\delta$-resolution determined, with $\delta$ as in Definition~\ref{d,a}.

\bigskip

We now give an example of a $(D,A)$-stacked algebra with  $D = 4$ and $A = 2$ which is not monomial.

\begin{example}\label{cube}
Let $\mathcal{Q}$ be the quiver given by
$$
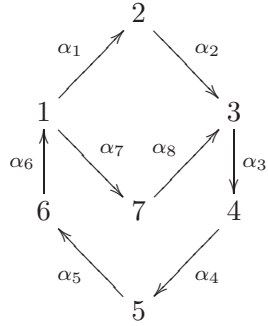
\xymatrix{
 & 2\ar[dr]^{\a_2} & \\
 1\ar[ur]^{\a_1} \ar[dr]^{\a_7} & & 3\ar[d]^{\a_3} \\
6\ar[u]^{\a_6} & 7\ar[ur]^{\a_8} & 4\ar[dl]^{\a_4}\\
 & 5\ar[ul]^{\a_5} & \\
}$$
and let $\Lambda = K\mathcal{Q} / I$, where $I = \langle (\a_1 \a_2 - \a_7 \a_8 )\a_3 \a_4 , \a_3 \a_4 \a_5 \a_6 , \a_5 \a_6 (\a_1 \a_2 - \a_7 \a_8 ) \rangle$.

The ideal $I$ is homogeneous with generators all of length 4. Thus $\Lambda$ is a length graded algebra. We show that $\Lambda$ is a $(D,A)$-stacked algebra with $D = 4$ and $A = 2$, by explicitly describing a minimal graded projective resolution of $\Lambda_0$.

Let $S_i$ be the simple module corresponding to the idempotent $e_i$, for $i = 1, \ldots, 7$.
It is easy to see that $S_2, S_4, S_6$ and $S_7$ have finite projective dimension, whilst $S_1, S_3$ and $S_5$ have
infinite projective dimension. Define projective $\Lambda$-modules $P^0 = \bigoplus_{i=1}^7 e_i \Lambda$,
$P^1 = \bigoplus_{i=1}^8 t(\alpha_i) \Lambda$ and $P^n = e_1 \Lambda \p e_3 \Lambda \p e_5 \Lambda$ for $n \geqslant 2$.
Let $d^0: P^0 \to \Lambda/\rrad$ be the canonical surjection, and
define $\L$-homomorphisms $d^n: P^n \to P^{n-1}$, for $ n\geqslant 1$, as follows:
\begin{enumerate}
\item[$\bullet$] $d^1: t(\alpha_i) \mapsto s(\alpha_i)\alpha_i t(\alpha_i)$\\
\item[$\bullet$] $d^2: \begin{cases}
e_1 \mapsto e_4\a_4\a_5\a_6 \\
e_3 \mapsto e_6\a_6(\a_1 \a_2 - \a_7 \a_8)\\
e_5 \mapsto e_2\a_2\a_3\a_4 - e_7\a_8\a_3\a_4\\
\end{cases}$\\
\item[$\bullet$] for $n \geqslant 3$, \ $d^n: \begin{cases}
e_1 \mapsto e_5\a_5 \a_6\\
e_3 \mapsto e_1(\a_1 \a_2 - \a_7 \a_8)\\
e_5 \mapsto e_3\a_3 \a_4\\
\end{cases}$\\
\end{enumerate}
where each path $p$ lies in the $s(p)\L$-component of $P^{n-1}$.
It is straightforward to verify that $(P^n, d^n)$ is a minimal graded projective resolution of $\L_0$ as a right $\L$-module.

It can be seen that the entries in a matrix representation of $d^0$ (resp. $d^1$, $d^2$) are of length 0 (respectively, 1, 3) so $P^0$ (respectively, $P^1$, $P^2$) is generated in degree 0 (respectively, 1, 4).
Moreover, for $n \geqslant 3$, the entries in a matrix representation of $d^n$ are of length $2$, so $P^n$ is generated in degree 2 more than that of $P^{n-1}$. Thus, for $n \geqslant 0$, the projective module $P^n$ is generated in degree
$$\begin{cases}
0 & \mbox{if $n = 0$} \\
1 & \mbox{if $n = 1$} \\
2n & \mbox{if $n \geqslant 2$}.
\end{cases} $$
Hence $\Lambda$ is a $(D,A)$-stacked algebra with $D = 4, A = 2$.
\end{example}

In order to prove that the Ext algebra of a $(D,A)$-stacked algebra is finitely generated, we require the following result from \cite{bib:GMMZ} which we state here for completeness.

\begin{prop}\label{sumof}\cite[Proposition 3.6]{bib:GMMZ}
Let $\Lambda=K\cQ/I$ be a finite-dimensional graded algebra with the length grading and let
$\cdots \rightarrow P^2 \rightarrow P^1 \rightarrow P^0 \rightarrow \Lambda_0 \rightarrow 0$
be a minimal graded projective resolution of $\Lambda_0$ as a right $\Lambda$-module.
Suppose that $P^i$ is finitely generated with generators in degree $d_i$, for $i = \a , \beta , \a + \beta$. Assume that  $d_{\a + \beta} = d_{\a} + d_{\beta}.$

Then the Yoneda map
$$\Ext^{\a}_{\Lambda} (\Lambda_0 , \Lambda_0 ) \times \Ext^{\beta}_{\Lambda} (\Lambda_0 , \Lambda_0 )
\rightarrow \Ext^{\a + \beta}_{\Lambda} (\Lambda_0 , \Lambda_0 )$$
is surjective.
Thus $$\Ext^{\a + \beta}_{\Lambda} (\Lambda_0 , \Lambda_0 )= \Ext^{\a}_{\Lambda} (\Lambda_0 , \Lambda_0 )\times
\Ext^{\beta}_{\Lambda} (\Lambda_0 , \Lambda_0 ) = \Ext^{\beta}_{\Lambda} (\Lambda_0 , \Lambda_0 ) \times
\Ext^{\a}_{\Lambda} (\Lambda_0 , \Lambda_0 ).$$
\end{prop}

\begin{thm}\label{genin}
Let $\Lambda = K\cQ / I$ be a $(D,A)$-stacked algebra with $D \geqslant 2$ and $A \geqslant 1$. Then $E(\Lambda)$ is generated in degrees $0, 1, 2$ and $3$.
\end{thm}

\begin{proof}
The algebra $\Lambda = K \cQ / I$ is a $(D,A)$-stacked algebra, so $I$ is generated by homogeneous elements of length $D$.
Thus $\Lambda$ is a graded algebra with the length grading. Each projective $P^n$ is generated in degree $\delta(n)$ where
$\delta$ is as in Definition~\ref{d,a}. Let $n\geqslant 4$ and recall that $\delta(2) = D$. For $n$ even, we have
$\delta(n-2) = (n-2)D/2$ and $\delta(n) =
nD/2$ so $\delta(2) + \delta(n-2) = \delta(n)$; similarly, for $n$ odd, we have
$\delta(n-2) = (n-3)D/2 + A$ and $\delta(n) = (n-1)D/2 + A$, so again we have $\delta(2) + \delta(n-2) = \delta(n)$.
Thus, from Proposition~\ref{sumof}, $\Ext^{n}_{\Lambda} (\Lambda_0, \Lambda_0 )= \Ext^{2}_{\Lambda} (\Lambda_0, \Lambda_0)\times \Ext^{n-2}_{\Lambda} (\Lambda_0, \Lambda_0 )$
for all $n\geqslant 4$. Thus, $E(\Lambda)$ is generated in degrees (at most) $0, 1, 2$ and $3$.
\end{proof}

The following corollary is now immediate.

\begin{cor}
Let $\Lambda = K\cQ / I$ be a $(D,A)$-stacked algebra with $D \geqslant 2$ and $A \geqslant 1$. Then $E(\Lambda)$ is a $\delta$-Koszul algebra.
\end{cor}

\section{Properties of the Ext algebra}\label{E(A)}

In this section we investigate some of the general properties of the Ext algebra of a $(D,A)$-stacked algebra and give a full characterisation of $(D,A)$-stacked algebras in Theorem~\ref{charac}.

We use the notation and results of \cite{bib:GMMZ} which we briefly recall here.
Let $\L = \L_0 \p \L_1 \p \cdots$ be a graded algebra, and let $M = \p_i M_i$ be a graded $\L$-module. The $n$th shift of $M$, denoted $M[n]$, is the graded $\L$-module $X = \p_i X_i$, where $X_i = M_{i-n}$. Hence, if $M$ is generated in degree $n$, then $M[-n]$ is generated in degree $0$. Let $\Gr(\L)$ be the category of graded right $\L$-modules together with the set of degree $0$ homomorphisms  and let $F$ denote the forgetful functor, $F: \Gr(\L) \to \Mod$-$\L$.

In addition to the grading of the Ext algebra by homological degree, namely
$E(\Lambda) = \bigoplus_{n\geqslant 0}\Ext^n_\Lambda(\Lambda_0,\Lambda_0)$, we also have the shift-grading.
Specifically, let $M = \p_i M_i$ and $N = \p_i N_i$ be graded $\L$-modules and suppose we have a graded projective resolution $(Q^{n},d^n)$ of $M$ where each $Q^n$ is finitely generated. Following \cite{bib:GMMZ}, we define $$\Ext^n_{\L} (F(M), F(N))_i = \Ext^n_{\Gr(\L)} (M, N[i]),$$ which is the homology of the complex obtained by applying $\Hom_{\Gr(\L)} (-,N[i])$ to $(Q^{n}, d^n)$, and is called the shift-grading.
In our setting, we have $M = N = \Lambda_0$ and a minimal graded projective resolution $(P^n, d^n)$ of $\Lambda_0$ as a $\L$-module. We can then apply \cite[Theorem 2.1]{bib:GMMZ} to see that
$$\Ext^n_{\L} (F(\Lambda_0), F(\Lambda_0))_i \cong \Hom_{\Gr(\L)} ( \Omega^n (\Lambda_0), \Lambda_0[i]) \cong \Hom_{\Gr(\L)} ( \Omega^n (\Lambda_0)[-i], \Lambda_0),$$ where $\Omega^n(\Lambda_0)$ denotes the $n$th syzygy of $\Lambda_0$ with respect to the resolution $(P^n, d^n)$.

We will use this to investigate when products in the Ext algebra are zero. We assume that $\L$ is a graded algebra and $(P^n, d^n)$ is a minimal graded projective resolution of $\L_0$.

\begin{prop}\label{1times1}
Let $\Lambda$ be a $(D,A)$-stacked algebra with $D > 2$. Then $$\Ext^{1}_{\Lambda} (\Lambda_0 , \Lambda_0) \times\Ext^{1}_{\Lambda} (\Lambda_0 , \Lambda_0) = 0.$$
\end{prop}

\begin{proof}
The projective module $P^1$ is generated in degree $1$ and
$\Ext^1_{\Lambda}(\Lambda_0 , \Lambda_0)  \cong \Hom_\L(P^1, \Lambda_0)$, so every element of
$\Ext^1_{\Lambda}(\Lambda_0, \Lambda_0)$ can be viewed as a short exact sequence of graded modules of the form
$0 \to \L_0[-1] \to E \to \L_0 \to 0$.
Let
\begin{equation}\label{P1}
0\rightarrow \Lambda_0 [-1] \rightarrow E \rightarrow \Lambda_0 \rightarrow 0
\end{equation}
and
\begin{equation}
\label{P2} 0\rightarrow \Lambda_0 [-1] \rightarrow \hat{E} \rightarrow \Lambda_0 \rightarrow 0
\end{equation}
be two short exact sequences in $\Ext^1_{\Lambda}(\Lambda_0, \Lambda_0)$.
We can shift the sequence \eqref{P2} by $-1$ to get
\begin{equation}\label{P3}
0\rightarrow \Lambda_0 [-2] \rightarrow \hat{E}[-1] \rightarrow \Lambda_0 [-1] \rightarrow 0.
\end{equation}
We then splice the sequences \eqref{P1} and \eqref{P3} together to obtain
$$0\rightarrow \Lambda_0 [-2] \rightarrow \hat{E}[-1]\rightarrow E \rightarrow \Lambda_0 \rightarrow 0.$$
Thus the image of $\Ext^1_{\Lambda}(\Lambda_0 , \Lambda_0)\times \Ext^1_{\Lambda}(\Lambda_0 , \Lambda_0)$ is contained
in $\Ext^2_{\Lambda}(\Lambda_0 , \Lambda_0)_2$. However, we know that $P^2$ is generated in degree $D \neq 2$, so
$\Ext^2_{\L} (\L_0, \L_0)_2 = 0$.
Therefore, $\Ext^1_{\Lambda}(\Lambda_0 , \Lambda_0)\times \Ext^1_{\Lambda}(\Lambda_0 , \Lambda_0) = 0$, when $D > 2$.
\end{proof}

\begin{prop}\label{ones}
Let $\Lambda$ be a $(D,A)$-stacked algebra with $D > 2$.
\begin{enumerate}[(i)]
\item If $D \neq A+1$ then $\Ext^{n}_{\Lambda} (\Lambda_0 , \Lambda_0) \times \Ext^{1}_{\Lambda} (\Lambda_0 , \Lambda_0) =
0 = \Ext^{1}_{\Lambda} (\Lambda_0 , \Lambda_0) \times\Ext^{n}_{\Lambda} (\Lambda_0 , \Lambda_0),$ for all $n$ odd, $n \geqslant 1$. \item If $A > 1$ then $\Ext^{n}_{\Lambda} (\Lambda_0 , \Lambda_0) \times \Ext^{1}_{\Lambda} (\Lambda_0 , \Lambda_0) =
0 = \Ext^{1}_{\Lambda} (\Lambda_0 , \Lambda_0) \times\Ext^{n}_{\Lambda} (\Lambda_0 , \Lambda_0),$  for all $n$ even, $n \geqslant 2$.
\end{enumerate}
\end{prop}

\begin{proof}
The case $n=1$ follows from Proposition \ref{1times1}. Thus we may assume $n \geqslant 2$.
The projective module $P^n$ is generated in degree $\delta (n)$. So, since $\Ext^{n}_{\L} ( \L_0 , \L_0 )  \cong \Hom_\L (P^n, \L_0)$, each extension can be viewed as an exact sequence of graded modules of the form $$ 0 \to \L_0[-\delta(n)] \to E_{n} \to \cdots \to E_1 \to \L_0 \to 0.$$
Using the shift-grading, we can shift this sequence by $-1$ to obtain $$0 \to \L_0[-\delta(n)-1] \to E_{n}[-1]  \to \cdots \to E_1[-1] \to \L_0[-1] \to 0.$$
We can then splice this with an extension $0 \to \L_0[-1] \to E' \to \L_0 \to 0$ from
$\Ext^1_{\L}(\L_0 , \L_0)$ to obtain
$$0 \to \L_0[-\delta(n)-1] \to E_{n}[-1] \to \cdots \to E_1[-1] \to E' \to \L_0 \to 0.$$
Thus the image of $\Ext^n_{\L}(\L_0 , \L_0) \times \Ext^1_{\L}(\L_0 , \L_0 )$ lies in $\Ext^{n+1}_{\L}(\L_0, \L_0)_{\delta(n)+1}$.
However, the projective $P^{n+1}$ is generated in degree $\delta(n+1)$, so
$\Ext^{n+1}_{\L}(\L_0, \L_0) = \Ext^{n+1}_{\L}(\L_0, \L_0)_{\delta(n+1)}$.

If $n = 2r+1$ is odd, then $\delta(n+1) = \delta (2r+2) = (r+1)D$ and $\delta (n) +1 = \delta(2r+1)+1 = rD + A + 1$,
and since $D \neq A+1$, we have $\Ext^{n+1}_{\L}(\L_0, \L_0)_{\delta(n) + 1} = 0$.
Hence $\Ext^n_{\L}(\L_0 , \L_0) \times \Ext^1_{\L}(\L_0 , \L_0) = 0$ in this case.

On the other hand, if $n=2r$ is even with $r \geqslant 1$, then $\delta(n+1) = \delta (2r+1) = rD +A $
and $ \delta(n)+1 = \delta(2r) + 1 = rD + 1$.
Since $A > 1$, we have $\Ext^{n+1}_{\L}(\L_0, \L_0)_{\delta(n) + 1} = 0$, and
hence $\Ext^n_{\L}(\L_0 , \L_0) \times \Ext^1_{\L}(\L_0 , \L_0) = 0$.

The case for  $\Ext^{1}_{\L}(\L_0 , \L_0) \times \Ext^{n}_{\L}(\L_0 , \L_0) = 0$ is similar.
\end{proof}

\begin{prop}\label{odd}
Let $\Lambda$ be a $(D,A)$-stacked algebra with $D > 2, D \neq 2A$. Then
$$\Ext^{2m+1}_{\Lambda} (\Lambda_0 , \Lambda_0) \times\Ext^{2n+1}_{\Lambda} (\Lambda_0 , \Lambda_0) = 0$$
for all $m, n \geqslant 1.$
\end{prop}

\begin{proof}
Let $m \geqslant 1, n \geqslant 1$. The projective modules $P^{2m+1}$ and $P^{2n+1}$ are generated in degrees
$\delta(2m+1)$ and $\delta(2n+1)$, respectively. So an extension in $\Ext^{2m+1}_{\Lambda}(\Lambda_0 , \Lambda_0)$ can be viewed
as an exact sequence of graded modules
\begin{equation}\label{P9}
0\rightarrow \Lambda_0[-\delta(2m+1)] \rightarrow E_{2m+1} \to \cdots \to E_1 \rightarrow \Lambda_0 \rightarrow 0
\end{equation}
and an extension in $\Ext^{2n+1}_{\Lambda}(\Lambda_0, \Lambda_0)$ can be viewed as an exact sequence of graded modules
\begin{equation}\label{P10}
0\rightarrow \Lambda_0[-\delta(2n+1)]\rightarrow  E'_{2n+1} \to \cdots \to E'_1 \rightarrow \Lambda_0 \rightarrow 0.
\end{equation}
Shifting the sequence \eqref{P10} by $-\delta(2m+1)$ and then
splicing with \eqref{P9}, gives
\begin{multline*}
0 \to \Lambda_0 [-\delta(2m+1)- \delta(2n+1)]\to  E'_{2n+1}[-\delta(2m+1)]\to \cdots \to  E'_1[-\delta(2m+1)]\\
\to E_{2m+1} \to \cdots \to E_1 \to \L_0 \to 0
\end{multline*}
which is in $\Ext^{2m+2n+2}_{\Lambda}(\Lambda_0, \Lambda_0)_{\delta(2m+1) + \delta(2n+1)}$.
However, $P^{2(m+n+1)}$ is generated in degree $\delta (2(m+n+1)) = (m+n+1)D$, whereas
$\delta(2m+1) + \delta(2n+1) = (m+n)D + 2A$. Since $D \neq 2A$, we have $\Ext^{2(m+n+1)}_{\L} (\L_0, \L_0 )_{\delta(2m+1) + \delta(2n+1)} = 0$, so $\Ext^{2m+1}_{\Lambda}(\Lambda_0 , \Lambda_0)\times \Ext^{2n+1}_{\Lambda}(\Lambda_0 , \Lambda_0)= 0$.
\end{proof}

We summarise Propositions~\ref{1times1}, \ref{ones} and \ref{odd} in the following result.

\begin{thm}\label{summary}
Let $\L$ be a $(D,A)$-stacked algebra with $D > 2$. Then
\begin{enumerate}[(i)]
\item $\Ext^{1}_{\L}(\L_0, \L_0) \times \Ext^{1}_{\L}(\L_0, \L_0)= 0$;
\item if $D \neq A+1$, then $\Ext^{n}_{\L}(\L_0, \L_0) \times \Ext^{1}_{\L}(\L_0, \L_0)= 0 =
\Ext^{1}_{\L}(\L_0, \L_0) \times \Ext^{n}_{\L}(\L_0, \L_0)$, for all $n$ odd, $n \geqslant 1$;
\item if $A > 1$, then $\Ext^{n}_{\L}(\L_0, \L_0) \times \Ext^{1}_{\L}(\L_0, \L_0)= 0 = \Ext^{1}_{\L}(\L_0, \L_0) \times \Ext^{n}_{\L}(\L_0, \L_0)$, for all $n$ even, $n \geqslant 2$;
\item if $D \neq 2A$, then $\Ext^{2m+1}_{\L}(\L_0, \L_0) \times \Ext^{2n+1}_{\L}(\L_0, \L_0)= 0$, for all $n,m \geqslant 1$.
\end{enumerate}
\end{thm}

We now use Theorems~\ref{genin} and \ref{summary} to give the following characterisation of $(D,A)$-stacked algebras. This uses the characterisation of $D$-Koszul algebras from \cite[Theorem~ 4.1]{bib:GMMZ}.

\begin{thm}\label{charac}
Let $\L = K\cQ / I$ where $I$ is generated by homogeneous elements of length $D \geqslant 2$. Then $\L = \L_0 \p \L_1 \p \cdots$ is length graded. Suppose, in the minimal projective resolution $(P^n , d^n)$ of $\L_0$ that $P^3$ is generated in a single degree, $D + A$, for $A \geqslant 1$. Then $\L$ is a $(D,A)$-stacked algebra if and only if $E(\L)$ is generated in degrees $0, 1, 2$ and $3$ and the following conditions hold:
\begin{enumerate}[(i)]
\item if $D > 2$ then $\Ext^1_{\L} (\L_0 , \L_0) \times \Ext^1_{\L} (\L_0 , \L_0) = 0$;
\item if $D > 2, D \neq A+1$, then $\Ext^{n}_{\L} (\L_0 , \L_0) \times \Ext^{1}_{\L} (\L_0 , \L_0) = 0 =  \Ext^{1}_{\L} (\L_0 , \L_0) \times \Ext^{n}_{\L} (\L_0 , \L_0)$, for all $n$ odd, $n \geqslant 1$;
\item if $D > 2, A > 1$, then $\Ext^{n}_{\L} (\L_0 , \L_0) \times \Ext^{1}_{\L} (\L_0 , \L_0) = 0 = \Ext^{1}_{\L} (\L_0 , \L_0) \times \Ext^{n}_{\L} (\L_0 , \L_0)$, for all $n$ even, $n \geqslant 2$;
\item if $D >2, D \neq 2A$, then $\Ext^{2m+1}_{\L} (\L_0 , \L_0) \times \Ext^{2n+1}_{\L} (\L_0 , \L_0) = 0$, for all $m, n \geqslant 1.$
\end{enumerate}
\end{thm}

\begin{proof}
Suppose $\L = K\cQ / I$ is a $(D,A)$-stacked algebra. Then from Theorem~\ref{genin} we know that $E(\L)$ is generated in degrees $0, 1, 2$ and $3$. From Theorem~\ref{summary} we know that conditions (i), (ii), (iii) and (iv) hold.

To show the other direction, we consider 3 cases.

\textbf{Case 1: $D = 2, A = 1.$}

Assume $I$ is generated by homogeneous elements of length $2$, $P^3$ is generated in degree $3$ and $E(\L)$ is generated in degrees $0, 1, 2$ and $3$. We know that $P^0$ is generated in degree $0$ and $P^1$ is generated in degree $1$, and, since $I$ is quadratic, we also have that $P^2$ is generated in degree $2$. By Proposition~\ref{sumof} with $\a = 1, \b = 1$ we have $\Ext^1_{\L} (\L_0, \L_0) \times \Ext^1_{\L} (\L_0, \L_0) = \Ext^2_{\L} (\L_0, \L_0)$, and, with $\a=1, \b=2$ we have $\Ext^1_{\L} (\L_0, \L_0) \times \Ext^2_{\L} (\L_0, \L_0) = \Ext^3_{\L} (\L_0, \L_0)$. Therefore $E(\L)$ is generated in degrees $0$ and $1$. Hence $\L$ is Koszul and therefore a $(2,1)$-stacked algebra.

\textbf{Case 2: $D > 2, A = 1.$}

Assume $I$ is generated by homogeneous elements of length $D$, $P^3$ is generated in degree $D+1$ and $E(\L)$ is generated in degrees $0, 1, 2$ and $3$. We know that $P^0$ is generated in degree $0$, $P^1$ is generated in degree $1$ and $P^2$ is generated in degree $D$. By Proposition~\ref{sumof} with $\a=1, \b=2$ we have $\Ext^1_{\L} (\L_0, \L_0) \times \Ext^2_{\L} (\L_0, \L_0) = \Ext^3_{\L} (\L_0, \L_0)$. Therefore $E(\L)$ is generated in degrees $0, 1$ and $2$ and by \cite[Theorem 4.1]{bib:GMMZ} $\L$ is $D$-Koszul. Hence $\L$ is a $(D,1)$-stacked algebra.

\textbf{Case 3: $D > 2, A > 1.$}

Suppose that $P^3$ is generated in degree $D+A$ and $E(\L)$ is generated in degrees $0, 1, 2$ and $3$, with conditions (i), (ii), (iii) and (iv) holding. We know that $\Ext^{0}_{\L}(\L_0, \L_0) = \Ext^{0}_{\L}(\L_0, \L_0)_{0}$ and $\Ext^{1}_{\L}(\L_0, \L_0) = \Ext^{1}_{\L}(\L_0, \L_0)_{1}$ since $P^0$ and $P^1$ are generated in degrees $0$ and $1$ respectively.
By hypothesis, we have $\Ext^{2}_{\L}(\L_0, \L_0) = \Ext^{2}_{\L}(\L_0, \L_0)_{D}$ and $\Ext^{3}_{\L}(\L_0, \L_0) = \Ext^{3}_{\L}(\L_0, \L_0)_{D+A}$.

We continue by induction to show that $\Ext^{n}_{\L}(\L_0 , \L_0) = \Ext^{n}_{\L}(\L_0 , \L_0)_{\delta(n)}$ for $n \geqslant 4$, where $\delta$ is as in Definition~\ref{d,a}.
We assume, for $m  < n$, that $\Ext^{m}_{\L}(\L_0 , \L_0) = \Ext^{m}_{\L}(\L_0 , \L_0)_{\delta(m)}$ and $P^m$ is generated in degree $\delta(m)$. We have
\begin{align*}
\Ext^n_{\L} (\L_0, \L_0) & = \Ext^{1}_{\L} (\L_0 , \L_0) \times \Ext^{n-1}_{\L} (\L_0 , \L_0)+ \Ext^{2}_{\L} (\L_0 , \L_0) \times \Ext^{n-2}_{\L} (\L_0 , \L_0)\\
& + \cdots + \Ext^{m}_{\L} (\L_0 , \L_0) \times \Ext^{n-m}_{\L} (\L_0 , \L_0)+ \cdots + \\
& \Ext^{n-2}_{\L} (\L_0 , \L_0)  \times \Ext^{2}_{\L} (\L_0 , \L_0) + \Ext^{n-1}_{\L} (\L_0 , \L_0) \times \Ext^{1}_{\L} (\L_0, \L_0).
\end{align*}

We consider the cases $n$ even and $n$ odd separately.

First we suppose that $n$ is even, $n \geqslant 4$ and consider $\Ext^{m}_{\L}(\L_0 , \L_0) \times \Ext^{n-m}_{\L} (\L_0 , \L_0)$.
If $m$ is even then $n-m$ is even, $P^m$ is generated in degree $\frac{m}{2}D$ and $P^{n-m}$ is generated in degree $\frac{(n-m)}{2}D$.
An element of $\Ext^{m}_{\L}(\L_0, \L_0)$ can be viewed as an exact sequence
\begin{equation}\label{25}
0 \to \L_0[-{\scriptstyle\frac{m}{2}}D ] \to E_{m} \to \cdots \to E_1 \to \L_0 \to 0
\end{equation}
and an element of $\Ext^{n-m}_{\L}(\L_0, \L_0)$ can be viewed as an exact sequence
\begin{equation}\label{26}
0 \to \L_0[-{\scriptstyle\frac{(n-m)}{2}}D] \to E'_{n-m} \to \cdots \to E'_1 \to \L_0 \to 0.
\end{equation}
Shifting \eqref{25} by $-{\scriptstyle\frac{(n-m)}{2}}D$ and then splicing the resulting sequence with \eqref{26} gives
\begin{multline*}
0 \to \L_0[-{\scriptstyle\frac{n}{2}}D ] \to E_m[-{\scriptstyle\frac{(n-m)}{2}}D] \to \cdots \to E_1[-{\scriptstyle\frac{(n-m)}{2}}D] \to \\
E'_{n-m} \to \cdots \to E'_1 \to \L_0 \to 0.$$
\end{multline*}
Thus the image of $ \Ext^{m}_{\L}(\L_0, \L_0) \times \Ext^{n-m}_{\L}(\L_0, \L_0)$ is contained in $\Ext^{n}_{\L}(\L_0, \L_0)_{\frac{n}{2}D}$.

Now suppose $m$ is odd with $m \geqslant 3$ and $n-m \geqslant 3$; then $n-m$ is also odd, $P^m$ is generated in degree $\frac{(m-1)}{2}D + A$ and $P^{n-m}$
is generated in degree $\frac{(n-m-1)}{2}D + A$. An element of $\Ext^{m}_{\L}(\L_0, \L_0)$ can be viewed as an exact sequence
\begin{equation}\label{20}
0 \to \L_0[-{\scriptstyle\frac{(m-1)}{2}}D - A] \to E_{m} \to \cdots \to E_1 \to \L_0 \to 0
\end{equation}
and an element of $\Ext^{n-m}_{\L} (\L_0 , \L_0)$ as an exact sequence
\begin{equation}\label{21}
0 \to \L_0[-{\scriptstyle\frac{(n-m-1)}{2}}D - A] \to E'_{n-m} \to \cdots \to E'_1 \to \L_0 \to 0.
\end{equation}
Shifting \eqref{20} by $-{\scriptstyle\frac{(n-m-1)}{2}}D - A$ and splicing the resulting sequence with \eqref{21} gives
\begin{multline*}
0 \to \L_0[-{\scriptstyle\frac{(n-2)}{2}}D - 2A] \to E_{m}[-{\scriptstyle\frac{(n-m-1)}{2}}D - A] \to \cdots \to E_1[-{\scriptstyle\frac{(n-m-1)}{2}}D - A] \to \\
E'_{n-m} \to \cdots \to E'_1 \to \L_0 \to 0.
\end{multline*}
Thus the image of  $ \Ext^{m}_{\L} (\L_0 , \L_0) \times \Ext^{n-m}_{\L} (\L_0 , \L_0)$ is contained in $ \Ext^{n}_{\L}(\L_0, \L_0)_{\frac{(n-2)}{2}D + 2A}.$

If $m = 1$ or $m = n-1$, then a similar argument shows that the images of $\Ext^{1}_{\L}(\L_0, \L_0) \times \Ext^{n-1}_{\L}(\L_0, \L_0)$ and $\Ext^{n-1}_{\L}(\L_0, \L_0) \times \Ext^{1}_{\L}(\L_0, \L_0)$ are both contained in $\Ext^{n}_{\L}(\L_0, \L_0)_{\frac{(n-2)}{2}D + A + 1}$.

We now use conditions (ii) and (iv) and show that some products are necessarily zero.
\begin{enumerate}
\item[$\bullet$] Suppose $D = 2A$. Then $D \neq A+ 1$, so by (ii), we have $\Ext^{1}_{\L}(\L_0, \L_0) \times \Ext^{n-1}_{\L}(\L_0, \L_0)$ $= 0 = \Ext^{n-1}_{\L}(\L_0, \L_0) \times \Ext^{1}_{\L}(\L_0, \L_0)$. The image of $\Ext^{m}_{\L}(\L_0, \L_0) \times \Ext^{n-m}_{\L}(\L_0, \L_0)$ is contained in $\Ext^{n}_{\L}(\L_0, \L_0)_{\frac{n}{2}D}$ for $m$ even and in $\Ext^{n}_{\L}(\L_0, \L_0)_{\frac{(n-2)}{2}D + 2A}$ for $m$ odd, $m \geqslant 3$. Since $D = 2A$, we have that the image of $\Ext^{m}_{\L}(\L_0, \L_0) \times \Ext^{n-m}_{\L}(\L_0, \L_0)$ is contained in $\Ext^{n}_{\L}(\L_0, \L_0)_{\frac{n}{2}D}$ for all $m \geqslant 1$. Hence $\Ext^{n}_{\L}(\L_0, \L_0)$ $= \Ext^{n}_{\L}(\L_0, \L_0)_{\frac{n}{2}D}$.
 \item[$\bullet$] Suppose $D \neq 2A, D = A + 1$. By (iv), $\Ext^{m}_{\L}(\L_0, \L_0) \times \Ext^{n-m}_{\L}(\L_0, \L_0) = 0$ for $m$ odd, $m \geqslant 3$. The images of $\Ext^{1}_{\L}(\L_0, \L_0) \times \Ext^{n-1}_{\L}(\L_0, \L_0)$ and $\Ext^{1}_{\L} (\L_0, \L_0) \times \Ext^{n-1}_{\L}(\L_0, \L_0)$ are contained  in $\Ext^{n}_{\L}(\L_0, \L_0)_{\frac{(n-2)}{2}D + A+ 1} = \Ext^{1}_{\L} (\L_0 , \L_0)_{\frac{n}{2}D}$, since $D = A + 1$. For $m$ even, the image of $\Ext^{m}_{\L}(\L_0, \L_0) \times \Ext^{n-m}_{\L}(\L_0, \L_0)$ is contained in $\Ext^{n}_{\L}(\L_0, \L_0)_{\frac{n}{2}D}$. Hence $\Ext^{n}_{\L} (\L_0 , \L_0) = \Ext^{n}_{\L} (\L_0 , \L_0)_{\frac{n}{2}D}$.
 \item[$\bullet$] Suppose $D \neq 2A, D \neq A + 1$. By (ii) and (iv), we have $\Ext^{m}_{\L}(\L_0, \L_0) \times \Ext^{n-m}_{\L} (\L_0 , \L_0)$ $= 0$, for $m$ odd, $m \geqslant 1$. The image of $\Ext^{m}_{\L}(\L_0, \L_0) \times \Ext^{n-m}_{\L}(\L_0, \L_0)$ is contained in $\Ext^{n}_{\L}(\L_0, \L_0)_{\frac{n}{2}D}$, for $m$ even. So $\Ext^{n}_{\L}(\L_0, \L_0) = \Ext^{n}_{\L}(\L_0, \L_0)_{\frac{n}{2}D}$.
\end{enumerate}
Hence for $n$ even, we have that $\Ext^{n}_{\L}(\L_0, \L_0) = \Ext^{n}_{\L}(\L_0, \L_0)_{\frac{n}{2}D}$. Thus $P^n$ is generated in degree $\frac{n}{2}D = \delta(n)$.

\medskip

Now suppose that $n$ is odd with $n \geqslant 5$. Then $n-1$ is even so, by (iii), we have that
$\Ext^{1}_{\L}(\L_0, \L_0) \times \Ext^{n-1}_{\L}(\L_0, \L_0) = 0 = \Ext^{n-1}_{\L}(\L_0, \L_0) \times
\Ext^{1}_{\L}(\L_0, \L_0)$. We now consider $\Ext^{m}_{\L}(\L_0, \L_0) \times \Ext^{n-m}_{\L}(\L_0, \L_0)$ for $m, n-m \geqslant 2$.
If $m$ is odd, then $n-m$ is even, $P^m$ is generated in degree $\frac{(m-1)}{2}D + A$ and $P^{n-m}$ is generated in
degree $\frac{(n-m)}{2}D$.
An element of $\Ext^{m}_{\L} (\L_0 , \L_0)$ has the form
\begin{equation}\label{16}
0 \to \L_0[-{\scriptstyle\frac{(m-1)}{2}}D - A] \to E_m \to \cdots \to E_1 \to \L_0 \to 0
\end{equation}
and an element of $\Ext^{n-m}_{\L} (\L_0 , \L_0)$ has the form
\begin{equation}\label{17}
0 \to \L_0[-{\scriptstyle\frac{(n-m)}{2}}D] \to E'_{n-m} \to \cdots \to E'_1 \to \L_0 \to 0.
\end{equation}
Shifting \eqref{16} by $-\frac{(n-m)}{2}D - A$ and splicing with \eqref{17} we obtain
\begin{multline*}
0 \to \L_0[-{\scriptstyle\frac{(n-1)}{2}}D - A] \to E_m[-{\scriptstyle\frac{(n-m)}{2}}D - A] \to \cdots E_1[-{\scriptstyle\frac{(n-m)}{2}}D - A] \to \\
E'_{n-m} \to \cdots \to E'_1 \to \L_0 \to 0.
\end{multline*}
Thus the image of $\Ext^{m}_{\L}(\L_0 ,\L_0) \times \Ext^{n-m}_{\L}(\L_0, \L_0)$ is contained in
$\Ext^{n}_{\L}(\L_0, \L_0)_{\frac{(n-1)}{2}D + A}$.

\sloppy If $m$ is even then $n-m$ is odd and a similar argument gives that
$\Ext^{m}_{\L}(\L_0, \L_0) \times \Ext^{n-m}_{\L}(\L_0, \L_0)$ is also contained in
$\Ext^{n}_{\L}(\L_0, \L_0)_{\frac{(n-1)}{2}D + A}$. Thus $\Ext^{n}_{\L}(\L_0, \L_0) = \Ext^{n}_{\L}(\L_0, \L_0)_{\frac{(n-1)}{2}D+A}$ and hence $P^n$ is generated in degree $\frac{(n-1)}{2}D + A = \delta(n)$.

Thus, for all $n \geqslant 0$, $P^n$ is generated in degree $\delta(n)$, where $\delta$ is as in Definition~\ref{d,a},
and hence $\L$ is a $(D,A)$-stacked algebra.
\end{proof}

\section{Regrading of the Ext Algebra}\label{regrad}

Under Koszul duality, the Ext algebra of a Koszul algebra is again a
Koszul algebra. For $D$-Koszul algebras, a regrading of the Ext algebra was introduced in \cite[Section 7]{bib:GMMZ}, called the hat-grading, in which the
regraded Ext algebra of a $D$-Koszul algebra is a Koszul algebra. Specifically, given a $D$-Koszul algebra $\L$, Green, Marcos, Mart\'inez-Villa and Zhang define the hat-grading of the Ext algebra $\hat{E}(\L)  = \p_{n \geqslant 0} \hat{E}(\L)_n$ by
$$\begin{array}{rcl}
 \hat{E}(\L)_0 & = & \Ext^0_{\Lambda} (\Lambda_0 , \Lambda_0)\\
 \hat{E}(\L)_1 & = & \Ext^1_{\Lambda} (\Lambda_0 , \Lambda_0)\p \Ext^2_{\Lambda} (\Lambda_0 , \Lambda_0)  \\
 \hat{E}(\L)_n & = & \Ext^{2n-1}_{\Lambda} (\Lambda_0 , \Lambda_0) \p \Ext^{2n}_{\Lambda} (\Lambda_0 , \Lambda_0) \quad \mbox{ for } n \geqslant 2,
\end{array}$$
and show in \cite[Theorem 7.1]{bib:GMMZ}, that $\hat{E}(\L)$ is a Koszul algebra with this regrading.
Recall that a $(D,A)$-stacked algebra which is neither Koszul nor $D$-Koszul necessarily has $D > 2$ and $A > 1$.
In this section, we define a grading on a $(D,A)$-stacked algebra with $D > 2, A > 1$, $D \neq 2A, D \neq A+1$ which we also call the hat-degree grading.
This will be used in Section~\ref{sec:extmonomial}, to show (with these conditions on $D, A$) that the regraded Ext algebra of a $(D,A)$-stacked monomial algebra is a Koszul algebra.

\begin{defin}\label{regrading}
Let $\L = K\cQ / I$ be a $(D,A)$-stacked algebra, with $D > 2 , A > 1$, $D \neq 2A$ and $D \neq A+1$. We define the hat-degree grading on the Ext algebra of $\Lambda$ by
$$\begin{array}{rcl}
\hat{E}(\L)_0 & = & \Ext^0_{\Lambda} (\Lambda_0 , \Lambda_0)\\
\hat{E}(\L)_1 & = & \Ext^1_{\Lambda} (\Lambda_0 , \Lambda_0)\p \Ext^2_{\Lambda} (\Lambda_0 , \Lambda_0) \p \Ext^3_{\Lambda} (\Lambda_0 , \Lambda_0) \\
\hat{E}(\L)_n & = & \Ext^{2n}_{\Lambda} (\Lambda_0 , \Lambda_0) \p \Ext^{2n+1}_{\Lambda} (\Lambda_0 , \Lambda_0) \quad \mbox{ for } n \geqslant 2.
\end{array}$$
Let $\hat{E}(\L) = \p_{n \geqslant 0} \hat{E}(\L)_n.$
\end{defin}

First we show that the hat-degree does indeed give a well defined grading.

\begin{thm}
Let $\Lambda$ be a $(D,A)$-stacked algebra with $D >2, A > 1$, $D \neq 2A, D \neq A+1$.  Then the Ext algebra $\hat{E}(\Lambda)$ is a graded algebra with the hat-degree grading of Definition~\ref{regrading}.
\end{thm}

\begin{proof}
We need to show $\hat{E}(\L)_m \times \hat{E}(\L)_n = \hat{E}(\L)_{m+n}$ for all $m,n \geqslant 0$.
This is clearly true if $m=0$ or $n=0$.
In the case $m = n = 1$ and using Theorem~\ref{summary}, we have
\begin{multline*}
\hat{E}(\L)_1 \times \hat{E}(\L)_1 =
\Ext^2_{\L} ( \L_0 , \L_0 ) \times \Ext^2_{\L} ( \L_0 , \L_0 )\p \Ext^2_{\L} ( \L_0 , \L_0 ) \times \Ext^3_{\L} ( \L_0 , \L_0 ) \p \\ \Ext^3_{\L} ( \L_0 , \L_0 ) \times \Ext^2_{\L} ( \L_0 , \L_0 ).
\end{multline*}
Now, by Proposition~\ref{sumof}, $\Ext^2_{\L} ( \L_0 , \L_0 ) \times \Ext^2_{\L} ( \L_0 , \L_0 ) = \Ext^4_{\L} ( \L_0 , \L_0 )$ and
$\Ext^2_{\L} ( \L_0 , \L_0 ) \times \Ext^3_{\L} ( \L_0 , \L_0 ) = \Ext^5_{\L} ( \L_0 , \L_0 ) = \Ext^3_{\L} ( \L_0 , \L_0 )
\times \Ext^2_{\L} ( \L_0 , \L_0 )$.
Thus $\hat{E} (\L)_1 \times \hat{E} (\L)_1  = \hat{E} (\L)_2$.
Now let $m=1, n \geqslant 2$. Using Proposition~\ref{sumof}, we have
\begin{multline*}
\hat{E} (\L)_1 \times \hat{E} (\L)_n
= \Ext^2_{\L} ( \L_0 , \L_0 )\times  \Ext^{2n}_{\L} ( \L_0 , \L_0 ) \p \Ext^{2}_{\L} ( \L_0 , \L_0 ) \times  \Ext^{2n+1}_{\L} ( \L_0 , \L_0 ) + \\
\Ext^{3}_{\L} ( \L_0 , \L_0 )\times  \Ext^{2n}_{\L} ( \L_0 , \L_0 ).
\end{multline*}
Theorem~\ref{summary} then gives $\hat{E} (\L)_1 \times \hat{E}(\L)_n = \hat{E} (\L)_{n+1}.$
Similar arguments show that $\hat{E}_n (\L) \times \hat{E}_1 (\L) = \hat{E}_{n+1}(\L)$ for $n \geqslant 2$, and that
$\hat{E} (\L)_m \times \hat{E} (\L)_n  = \hat{E} (\L)_{m+n}$ for $m, n \geqslant 2$.
\end{proof}

We now discuss the case where $D=2A$ and show that there is no regrading of the Ext algebra if $\gldim \L \geqslant 6$.

\begin{prop}\label{NOregrading}
Let $\L$ be a $(D, A)$-stacked algebra, with $D=2A, A > 1$ and $\gldim \L \geqslant 6$. Then there is no regrading such that the Ext algebra of $\L$ is Koszul.
\end{prop}

\begin{proof}
For $E(\L)$ to be Koszul we would need a hat-degree such that $\hat{E}(\L)$ is generated in degrees $0$ and $1$. From Theorem \ref{genin}, we know that $E(\L)$ is generated in degrees $0, 1, 2, 3$. Now, $\Ext^1_{\L} (\L_0, \L_0) \times \Ext^2_{\L} (\L_0, \L_0) =0$ from Theorem~\ref{summary}(iii), so $E(\L)$ cannot be generated in degrees $0, 1$ and $2$. So we may assume that $\hat{E}_0 (\L) = \Ext^0_{\L} (\L_0, \L_0)$ and $\hat{E}_1 (\L)  = \Ext^1_{\L} (\L_0, \L_0) \p \Ext^2_{\L} (\L_0, \L_0) \p \Ext^3_{\L} (\L_0, \L_0)$.
From Proposition~\ref{sumof} together with $D = 2A$, we have $\Ext^6_{\L} (\L_0, \L_0)  = $
$$\Ext^3_{\L} (\L_0, \L_0) \times \Ext^3_{\L} (\L_0, \L_0) =  \Ext^2_{\L} (\L_0, \L_0) \times \Ext^2_{\L} (\L_0, \L_0)  \times \Ext^2_{\L} (\L_0, \L_0).$$
Since $\gldim \L \geqslant 6$, we have that $\Ext^6_{\L} (\L_0, \L_0) \neq 0$. So we may choose $0 \neq z \in \Ext^6_{\L} (\L_0, \L_0)$ with $z = x_1 x_2 x_3 = \sum y_i y_i'$, for  $x_j \in \Ext^2_{\L} (\L_0, \L_0)$ and $y_i, y_i' \in \Ext^3_{\L} (\L_0, \L_0)$. But $x_j, y_i, y_i' \in \hat{E}_1 (\L)$ so $\sum y_i y_i' \in (\hat{E}_1 (\L))^2$ whereas $x_1 x_2 x_3 \in (\hat{E}_1 (\L))^3$. Thus any definition of a hat-degree would require $z \in \hat{E}_2 (\L)\cap \hat{E}_3 (\L)$ which contradicts the definition of a grading.
\end{proof}

The hypothesis $\gldim\L \geqslant 6$ in the previous result is necessary, but to see this we require some Gr\"{o}bner basis theory to show that the Ext algebra is Koszul. We return to this in Example~\ref{example gldim 6}.

\section{A review of Gr\"obner basis theory}\label{sec:GB}

We give a brief introduction to Gr\"{o}bner bases following \cite{bib:FFG}, \cite{bib:G} and \cite{bib:GH}.

Let $\Gamma$ be a finite quiver, and let $\mathcal{B}$ be the basis of the path algebra $K\Gamma$ which consists of all paths in $K\Gamma$. We remark that $\mathcal{B}$ is a multiplicative basis of $K\Gamma$, that is, if $p, q \in \mathcal{B}$ then either $pq \in \mathcal{B}$ or $pq = 0$. An admissible order on $\mathcal{B}$ is a well-order $>$ on $\mathcal{B}$ that satisfies the following properties:
\begin{enumerate}
\item if $p, q, r \in \mathcal{B}$ and $p > q$ then $pr > qr$ if both are not zero and $rp > rq$ if both are not zero;
\item if $p, q, r \in \mathcal{B}$ and $p = qr$ then $p \geq q$ and $p \geq r$.
\end{enumerate}
From \cite{bib:G}, the left length lexicographic order is an admissible order and is defined as follows.
Arbitrarily order the vertices and arrows such that every vertex is less than every arrow. For paths
of length greater than $1$, if $p = \a_1 \a_2 \cdots \a_n$ and $q = \b_1 \b_2 \cdots \b_m$ where the
$\a_i$ and $\b_i$ are arrows and $p,q \in \mathcal{B}$, then $p > q$ if $ n > m$ or, if $n = m$, then
there is some $1 \leqslant i \leqslant n$ with $\a_j = \b_j $ for $j < i$ and $\a_i > \b_i$.

Let $K\Gamma$ be a path algebra and let $\mathcal{B}$ be the basis of all paths, with admissible order~$>$.
Let $x$ be an element of $K\Gamma$, so $x$ is a linear combination of paths $p_1, \dots, p_n$ in $\mathcal{B}$
(all with non-zero coefficients). Then $\Tip(x)$ is the largest $p_i$ occurring in $x$ with respect
to the ordering $>$. We let $\CTip(x)$ denote the coefficient of $\Tip(x)$. The paths $p_1, \dots, p_n$ are called the support of $x$, denoted $\Supp(x)$.
If $I$ is an ideal in $K\Gamma$ then $\Tip(I)$ is the set of paths that occur as tips of non-zero elements of $I$. We let Nontip$(I)$ be the set of finite paths in $K\Gamma$ that are not in $\Tip(I)$. A Gr\"{o}bner basis for $I$ is a non-empty subset $G$ of $I$ such that the tip of each non-zero element of $I$ is divisible by the tip of some element in~$G$.

Now let $a$ be a non-zero element of $K\Gamma$. A simple (algebra) reduction for $a$ is determined by a $4$-tuple
$(\l, u, f, v)$ where
$\l \in K\backslash\{0\}, f \in K\Gamma \backslash \{0\}$ and $u,v \in \mathcal{B}$, satisfying:
\begin{enumerate}
\item $u \Tip(f) v \in \Supp(a)$;
\item $u \Tip(f) v \notin \Supp(a - \l u f v)$.
\end{enumerate}
We say that $a$ reduces over $f$ to $a - \l u f v$ and write $a \Rightarrow_f a - \l u f v$. More generally, $a$ reduces to $a'$ over a set $X = \{f_1, \ldots, f_n \}$, denoted $a \Rightarrow_X a'$, if there is a finite sequence of reductions so that $a$ reduces to $a_1$ over $f_1$, $a_i$ reduces to $a_{i+1}$ over $f_{i+1}$ for $i = 1, \dots, n-2$, and $a_{n-1}$ reduces to $a'$ over $f_n$.

For $a, b \in \mathcal{B}$, we say that $a$ divides $b$ if there exist $u, v \in \mathcal{B}$ such that $b = uav$. Let $h_1, h_2 \in K\Gamma$ and suppose there are elements $p, q \in \mathcal{B}$ such that:
\begin{enumerate}
\item $\Tip(h_1) p = q \Tip(h_2)$;
 \item $\Tip(h_1) $ does not divide $q$ and $\Tip(h_2)$ does not divide $p$.
\end{enumerate}
Then the overlap difference of $h_1$ and $h_2$ by $p, q$ is defined to be
$$o(h_1, h_2, p, q) = (1 / \CTip(h_1)) h_1 p - (1 / \CTip(h_2)) q h_2.$$
It is clear that the overlap difference of two elements in $\mathcal{B}$ is always zero.

\bigskip

We are now ready to state the main results we will need in Sections~\ref{sec:extmonomial} and \ref{sec:notmonomial}. We keep the above notation.

\begin{thm}\cite[Theorem 13]{bib:FFG}\label{reducedG}
Let $K\Gamma$ be a path algebra and let $\mathcal{H}  = \{h_j : j \in \mathcal{J}\}$ be a set of non-zero uniform elements
in $K\Gamma$ which generates the ideal $I$. Assume that the following conditions hold, for all $i, j \in \mathcal{J}$:
\begin{enumerate}[(i)]
\item $\CTip(h_j) = 1$;
\item $h_i$ does not reduce over $h_j$ if $i \neq j$;
\item every overlap difference for two (not necessarily distinct) members of $\mathcal{H}$ always reduces to zero over $\mathcal{H}$.
\end{enumerate}
Then $\mathcal{H}$ is a reduced Gr\"{o}bner basis of $I$.
\end{thm}

\begin{thm}\cite[Theorem 3]{bib:GH}\label{Kosz}
Let $K\Gamma$ be a path algebra, let $I$ be a quadratic ideal of $K\Gamma$ and set $A = K\Gamma / I$. Then:
\begin{enumerate}[(i)]
\item The reduced Gr\"{o}bner basis of $I$ consists of homogeneous elements.
\item If the reduced Gr\"{o}bner basis of $I$ consists of quadratic elements then $A$ is a Koszul algebra.
\end{enumerate}
\end{thm}

\section{Koszulity of the Ext algebra of a $(D,A)$-stacked monomial algebra}\label{sec:extmonomial}

In this section we show that there is a regrading so that the Ext algebra of a $(D,A)$-stacked monomial algebra is a Koszul algebra, providing we have $D \neq 2A$ when $A > 1$.

Let $\Lambda$ be a $(D,A)$-stacked monomial algebra. In the case where $D=2$ then $\Lambda$ is a Koszul algebra, in which case it is known that the Ext algebra is again Koszul, so no regrading of the Ext algebra is required. Moreover, the structure of $E(\Lambda)$ by quiver and relations was given in \cite{bib:GMVII}. If $D>2$ and $A=1$ then $\Lambda$ is a $D$-Koszul algebra, and it was shown in  \cite[Theorem 7.1]{bib:GMMZ} that there is a regrading of the Ext algebra under which the regraded Ext algebra is a Koszul algebra.

Thus we may assume that $\Lambda$ is a $(D,A)$-stacked monomial algebra with $D>2, A>1$ and $D \neq 2A$; we use the regrading of Section~\ref{regrad}. We note that if $\gldim \Lambda \geqslant 4$, then \cite[Proposition 3.3(3)]{bib:GS2} shows that $A | D$.
Thus $D \neq A+1$ (for otherwise $D = dA$ so that $dA = A+1$, which cannot occur since $A > 1$).
So the conditions of Definition~\ref{regrading} are satisfied and we can consider the hat-degree grading
on $E(\L)$. On the other hand, if $\gldim \Lambda < 4$ then
$E(\L) = \bigoplus_{i=0}^3\Ext^i_0(\L_0, \L_0)$ and
the definitions of $\hat{E}(\L)_0$ and $\hat{E}(\L)_1$ in Definition~\ref{regrading} trivially give
a grading in this case.

The Ext algebra of a monomial algebra was described by Green and Zacharia in \cite[Theorem B]{bib:GZ},
using the concept of overlaps from \cite{GHZ}. We start by recalling their work before
turning our attention to the case of $(D,A)$-stacked monomial algebras. (The reader may also see \cite{bib:GS3}.)

Let $\Lambda = K{\mathcal Q}/I$ be a monomial algebra. Fix a set ${\R}^2$ which is a minimal generating
set of paths for $I$. Recall that a path $p$ is a {\it prefix} of a path $q$ if there is some
path $p'$ such that $q = pp'$.

\begin{defin}
\begin{enumerate}
\item A path $q$ overlaps a path $p$ with overlap $pu$ if there are paths
$u$ and $v$ such that $pu = vq$ and $1 \leqslant \ell(u) < \ell(q)$. We
may illustrate the definition with the following diagram.
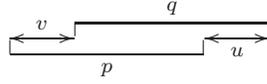
\[\xymatrix@W=0pt@M=0.3pt{
\ar@{^{|}-^{|}}@<-1.25ex>[rrr]_p\ar@{{<}-{>}}[r]^{v} &
\ar@{_{|}-_{|}}@<1.25ex>[rrr]^q & & \ar@{{<}-{>}}[r]_{u} &
& }\]
Note that we allow $\ell(v) = 0$ here.
\item A path $q$ properly overlaps a path $p$ with overlap $pu$ if $q$ overlaps $p$ and
$\ell(v) \geqslant 1$.
\item A path $p$ has no overlaps with a path $q$ if $p$ does not properly overlap $q$
and $q$ does not properly overlap $p$.
\end{enumerate}
\end{defin}

We use these overlaps to define sets $\R^n$ recursively. Let
$$\begin{array}{lll}
\R^0 & = & \mbox {the set of vertices of $K{\mathcal Q}$,}\\
\R^1 & = & \mbox{the set of arrows of $K{\mathcal Q}$,}\\
\R^2 & = & \mbox{the minimal generating set of paths for $I$.}
\end{array}$$
For $n \geqslant 3$, we say $R^2 \in \R^2$ maximally overlaps $R^{n-1} \in \R^{n-1}$
with overlap $R^n = R^{n-1}u$ if
\begin{enumerate}
\item $R^{n-1} = R^{n-2}p$ for some path $p$;
\item $R^2$ overlaps $p$ with overlap $pu$;
\item no proper prefix of $pu$ is an overlap of an element of $\R^2$ with $p$.
\end{enumerate}
We may also say that $R^n$ is a maximal overlap of $R^2 \in \R^2$ with $R^{n-1} \in \R^{n-1}$.

We let $\R^n$ denote the set of all overlaps $R^n$ formed in this way.

The construction of the paths in $\R^n$ may be illustrated with the following diagram of $R^n$.
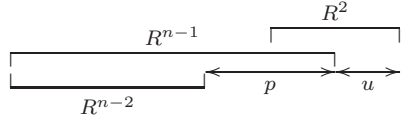
\[\xymatrix@W=0pt@M=0.3pt{
\ar@{^{|}-^{|}}@<-1.25ex>[rrr]_{R^{n-2}}\ar@{_{|}-_{|}}@<1.5ex>[rrrrr]^{R^{n-1}}
& & & \ar@{{<}-{>}}[rr]_{p} & \ar@{_{|}-_{|}}@<3.5ex>[rr]^{R^2} &
\ar@{{<}-{>}}[r]_{u} & & }\]

Recall from \cite{GHZ} that if $R_1^np = R_2^nq$, for $R_1^n,
R_2^n \in \R^n$ and paths $p, q$, then $R_1^n = R_2^n$ and $p =
q$.

These sets $\R^n$ were used in \cite{GHZ} to construct a minimal graded projective
resolution $(P^n, d^n)$ of $\L_0$ as a right $\L$-module as follows. Define $P^n = \bigoplus_{R^n \in
\R^n}t(R^n)\L$ for all $n\geqslant 0$. Let $d^0: P^0 \to \L_0$ be the canonical surjection and,
for $n \geqslant 1$, define $\L$-homomorphisms $d^n: P^n \to P^{n-1}$ by
$t(R^n) \mapsto t(R^{n - 1})p$ where $R^n = R^{n-1}p$ and $t(R^{n - 1})p$ is in the component of
$P^{n-1}$ corresponding to $t(R^{n-1})$.

We identify the set $\R^n$ with a basis $f^n$ for $\Ext_\L^n(\L_0, \L_0)$ in the following way.
List the elements of $\R^n$ as $R^n_1, \dots , R^n_r$ for some $r$. Let $n \geqslant 0$ and define
$f^n_i$ to be the $\L$-homomorphism $P^n \to \L_0$ given by
$$f^n_i : t(R^n_j) \mapsto
\left \{ \begin{array}{ll}
t(R^n_i) + \rrad & \mbox{ if }  i = j \\
0 & \mbox{ otherwise}.
\end{array} \right. $$
We use right modules throughout this paper, together with the standard convention that we compose module homomorphisms from right to left. Thus the composition $f \circ g$ means we apply $g$ first then $f$.
Recall that we write paths in a quiver from left to right.
So, if $f^n_i$ corresponds to the path $R^n_i \in \R^n$ and if $R^n_i = eR^n_i e'$, where $e =s(R^n_i)$ and $e' = t(R^n_i)$ are in $\R^0$, then
$f^n_i = f^0_{e'}f^n_if^0_e$ where $f^0_e$ (respectively, $f^0_{e'}$) denotes the element of $f^0$ that corresponds to $e$ (respectively, $e'$).
With this notation, we have from \cite{bib:GZ} that $f^m_j f^n_i \neq 0$ in $E(\L)$ if and only if $R^n_iR^m_j  = R^{n+m}_k \in \R^{n+m}$ for some $k$ and $f^m_j f^n_i = f^{n+m}_k$.

\bigskip

Now let $\Lambda = K{\mathcal Q}/I$ be a $(D,A)$-stacked monomial algebra with $D>2, A> 1$ and $D \neq 2A$.
The set $f^0$ is a basis of $\hat{E}(\L)_0$ and the set $f^1\cup f^2\cup f^3$ is a basis of $\hat{E}(\L)_1$.
From Theorem~\ref{genin}, we have that $E(\Lambda)$ is generated in degrees 0, 1, 2 and 3, and, from Theorem~\ref{summary}, we know $\Ext^1_{\L} (\L_0 , \L_0) \times \Ext^1_{\L} (\L_0 , \L_0) = 0$ and $\Ext^1_{\L} (\L_0 , \L_0) \times \Ext^2_{\L} (\L_0 , \L_0) = 0 = \Ext^2_{\L} (\L_0 , \L_0) \times \Ext^1_{\L} (\L_0 , \L_0) = 0$. Thus, the set $f^0\cup f^1\cup f^2\cup f^3$ is a minimal generating set for $E(\L)$ as a $K$-algebra.

Let $\Delta$ be the quiver with vertex set $\Delta_0 = f^0$ and arrow set $\Delta_1 = f^1\cup f^2\cup f^3$.
Following \cite[Definition 2.1]{bib:GZ}, let $I_\Delta$ be the ideal of $K\Delta$ generated by elements
of the form $f^{n_1}_{i_1} \cdots f^{n_r}_{i_r}$ where the path $R^{n_r}_{i_r} \cdots R^{n_1}_{i_1}$
is not in $R^{n_1 + \cdots + n_r}$ and by elements of the form
$f^{n_1}_{i_1} \cdots f^{n_r}_{i_r} - f^{m_1}_{i_1} \cdots f^{m_s}_{i_s}$ where, as paths in
$K{\mathcal Q}$, we have $R^{n_r}_{i_r} \cdots R^{n_1}_{i_1} = R^{m_s}_{i_s} \cdots R^{m_1}_{i_1}$
and $n_1 + \cdots + n_r = m_1 + \cdots + m_s$. The Ext algebra of our $(D,A)$-stacked monomial algebra $\Lambda$ can now be described by quiver and relations.

\begin{thm}\cite[Theorem B]{bib:GZ}
Let $\Lambda = K{\mathcal Q}/I$ be a $(D,A)$-stacked monomial algebra with $D>2, A> 1$ and $D \neq 2A$. Keeping the above notation, we have $\hat{E}(\L) \cong K\Delta/I_\Delta$.
\end{thm}

We now describe a minimal generating set $\mathcal{H}_\Delta$ for $I_\Delta$. The aim is to show in Proposition~\ref{prop:monredGB}, that it is also a quadratic reduced Gr\"{o}bner basis of $I_\Delta$, and hence use Theorem~\ref{Kosz} to show that $K\Delta/I_\Delta$ is a Koszul algebra. Note that since $\Lambda$ is a $(D,A)$-stacked monomial algebra, each element of $\R^n$ has length $\delta(n)$ where $\delta$ is as in Definition~\ref{d,a}. We start with the following proposition.

\begin{prop}\label{prop:R5}
Let $\Lambda = K{\mathcal Q}/I$ be a $(D,A)$-stacked monomial algebra with $D>2, A> 1$ and $D \neq 2A$.
\begin{enumerate}[(i)]
\item If $R^2_iR^3_j \in \R^5$ for some $i, j$, then $R^2_iR^3_j = R^3_kR^2_l$ for some $R^2_l \in \R^2, R^3_k \in \R^3$.
\item If $R^3_iR^2_j \in \R^5$ for some $i, j$, then $R^3_iR^2_j = R^2_kR^3_l$ for some $R^2_k \in \R^2, R^3_l \in \R^3$.
\end{enumerate}
\end{prop}

\begin{proof}
We prove the first statement and leave the second to the reader.
Without loss of generality, we may assume that $R^2_1R^3_1 \in \R^5$.
Then $R^2_1R^3_1 = R^5_k \in \R^5$ for some $k$.
The element $R^3_1$ is a maximal overlap of
$R^2_2 \in \R^2$ with $R^2_3 \in \R^2$ so that we may illustrate
$R^3_1$ as follows:
$$
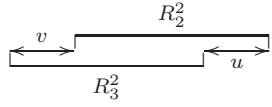
\xymatrix@W=0pt@M=0.3pt{
\ar@{^{|}-^{|}}@<-1.25ex>[rrr]_{R^2_3}\ar@{{<}-{>}}[r]^{v} &
\ar@{_{|}-_{|}}@<1.25ex>[rrr]^{R^2_2} & & \ar@{{<}-{>}}[r]_{u} & &
}$$
From the construction of $R^5_k \in \R^5$ via overlaps, there is some $R^3_i \in \R^3$ which is a prefix of $R^5_k$, so we may write $R^5_k = R^3_ip$ for some path $p$. Now, $R^5_k$ has length $2D+A$ and $R^3_i$ has length $D+A$, so $p$ has length $D$. However, $R^5_k = R^2_1R^3_1 = R^2_1vR^2_2$ as paths in $K\cQ$ and $R^2_2$ has length $D$, so we have that $p = R^2_2$ and thus $R^5_k = R^3_iR^2_2$ as required.

Note that we may illustrate the element $R^5_k$ as follows:
$$
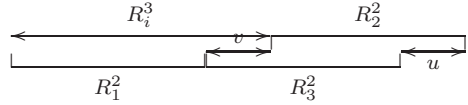
\xymatrix@W=0pt@M=0.3pt{
\ar@{^{|}-^{|}}@<-1.25ex>[rrr]_{R^2_1}\ar@{{<}-{>}}@<1.25ex>[rrrr]^{R^3_i} & & &\ar@{^{|}-^{|}}@<-1.25ex>[rrr]_{R^2_3}\ar@{{<}-{>}}[r]^{v} &
\ar@{_{|}-_{|}}@<1.25ex>[rrr]^{R^2_2} & & \ar@{{<}-{>}}[r]_{u} & &
}$$
\end{proof}

For ease of notation, we now use $a, b, c$ to denote the elements of $f^1, f^2, f^3$ respectively.
It follows from Proposition~\ref{prop:R5}, for $b \in f^2$ and $c\in f^3$, that if $bc \neq 0$ in $\Ext^5_\L(\L_0, \L_0)$ then there is some $b' \in f^2$ and $c'\in f^3$ such that $bc=c'b'$ in $\Ext^5_\L(\L_0, \L_0)$. Similarly, if $cb \neq 0$ in $\Ext^5_\L(\L_0, \L_0)$ then there is some $b' \in f^2$ and $c'\in f^3$ such that $cb=b'c'$ in $\Ext^5_\L(\L_0, \L_0)$.

\begin{defin}
Let $\Lambda = K{\mathcal Q}/I$ be a $(D,A)$-stacked monomial algebra with $D>2, A> 1$ and $D \neq 2A$.
Let $K\Delta$ be as above.
Let $\mathcal{H}_\Delta$ be the subset of $K\Delta$ containing all elements of the form
$$aa',\ ab,\ ba,\ ac,\ ca,\ cc'$$
together with
$$\begin{cases}
bb' & \mbox{if $bb' = 0$ in $\Ext^4_\L(\L_0, \L_0)$}\\
bc & \mbox{if $bc = 0$ in $\Ext^5_\L(\L_0, \L_0)$}\\
cb & \mbox{if $cb = 0$ in $\Ext^5_\L(\L_0, \L_0)$}\\
bc - c'b' & \mbox{if $bc \neq 0$ and where $bc=c'b'$ in $\Ext^5_\L(\L_0, \L_0)$}
\end{cases}$$
for all $a, a' \in f^1,\ b, b' \in f^2$ and $c, c' \in f^3$.
\end{defin}

\begin{prop}\label{prop:monredgen}
Let $\Lambda = K{\mathcal Q}/I$ be a $(D,A)$-stacked monomial algebra with $D>2, A> 1$ and $D \neq 2A$. Keeping the above notation, the set $\mathcal{H}_\Delta$ is a minimal generating set for $I_\Delta$.
\end{prop}

\begin{proof}
We have $D>2, A> 1$ and $D \neq 2A$ and, from the discussion at the start of this section, we may also assume that $D \neq A+1$. So, from Theorem~\ref{summary}, $aa', ab, ba, ac, ca, cc'$ are in $I_\Delta$ for all $a, a' \in f^1, b, b' \in f^2$ and $c, c' \in f^3$. Thus the ideal $\langle\mathcal{H}_\Delta\rangle$ is contained in $I_\Delta$.

The next step is to give two expressions for an arbitrary element $\eta$ of $f^n$ where
$n \geqslant 2$.
We may write $\eta$ as a finite product of monomials in the
generators $f^1\cup f^2 \cup f^3$. Moreover, $\eta$ is a non-zero element in $\Ext^n_\L(\L_0, \L_0)$ so $\eta \not\in \langle\mathcal{H}_\Delta\rangle$.
For ease of notation, let $b$ denote an arbitrary element of $f^2$ so that $b^r$ denotes a product
of $r$ elements in $f^2$.
Since $aa', ab, ba, ac, ca, cc' \in \mathcal{H}_\Delta$, we have that
$\eta = b^{t_1}c^{\varepsilon_1}b^{t_2}c^{\varepsilon_2} \cdots b^{t_r}c^{\varepsilon_r}$
for some $t_j \geqslant 0$ and $\varepsilon_j \in \{0, 1\}$.
Suppose $\varepsilon_1 = 1$ and $t_2 \geqslant 1$. Then $cb \neq 0$ in $\Ext^5_\L(\L_0, \L_0)$
so, by Proposition~\ref{prop:R5}, we may write $cb = b'c'$ for some $b'\in f^2, c'\in f^3$.
Thus $\eta = b^{t_1}(b'c' - h)b^{t_2-1}c^{\varepsilon_2} \cdots b^{t_r}c^{\varepsilon_r}$
where $h = b'c' - cb \in \mathcal{H}_\Delta$.
Thus $\eta = b^{t_1+1}c'b^{t_2-1}c^{\varepsilon_2} \cdots b^{t_r}c^{\varepsilon_r} + {\mathfrak h}$ for some ${\mathfrak h} \in \langle\mathcal{H}_\Delta\rangle$.
Continued use of Proposition~\ref{prop:R5} and that $cc' \in \mathcal{H}_\Delta$, shows that we may write
$\eta = b^tc^{\varepsilon} + {\mathfrak h}$ where $t = t_1+ \cdots + t_r$, $\varepsilon \in \{0, 1\}$ and ${\mathfrak h} \in \langle\mathcal{H}_\Delta\rangle$.
Similarly, we may write $\eta$ in the form $c^{\varepsilon}b^t + {\mathfrak h}$ where $\varepsilon \in \{0, 1\}$ and ${\mathfrak h} \in \langle\mathcal{H}_\Delta\rangle$.

To show the reverse inclusion that $I_\Delta \subseteq \langle\mathcal{H}_\Delta\rangle$, we suppose
first that $f^{n_1}_{i_1} \cdots f^{n_r}_{i_r}$ is in $I_\Delta$.
Without loss of generality, we may assume that $f^{n_1}_{i_1}f^{n_2}_{i_2} \in I_\Delta$ with $f^{n_1}_{i_1}, f^{n_2}_{i_2}$ both non-zero in $E(\Lambda)$.
Thus $R^{n_1}_{i_1} \in \R^{n_1}, R^{n_2}_{i_2}\in \R^{n_2}$ but the path $R^{n_2}_{i_2}R^{n_1}_{i_1}$
is not in $R^{n_1 + n_2}$.
We consider the case where $n_1, n_2 \geqslant 2$; the other cases are straightforward and are left to the reader.
From the argument above, we may write $f^{n_1}_{i_1} = b^{t_1}c^{\varepsilon_1} + {\mathfrak h_1}$ and
$f^{n_2}_{i_2} = c^{\varepsilon_2}b^{t_2} + {\mathfrak h_2}$,
where $\varepsilon_1, \varepsilon_2 \in \{0, 1\}$ and ${\mathfrak h}_1, {\mathfrak h}_2 \in \langle\mathcal{H}_\Delta\rangle$.
Then $f^{n_1}_{i_1}f^{n_2}_{i_2} = b^{t_1}c^{\varepsilon_1}c^{\varepsilon_2}b^{t_2} + {\mathfrak h}$ for some
${\mathfrak h} \in \langle\mathcal{H}_\Delta\rangle$.
Note that the path corresponding to $b^{t_1}c^{\varepsilon_1}$ is in $\R^{n_1}$
so that $n_1 = 2t_1 + 3\varepsilon_1$, the path corresponding to $c^{\varepsilon_2}b^{t_2}$
is in $\R^{n_2}$ so that $n_2 = 2t_2 + 3\varepsilon_2$, but the path corresponding to
$b^{t_1}c^{\varepsilon_1}c^{\varepsilon_2}b^{t_2}$ is not in $\R^{n_1+n_2}$.
We now show that $b^{t_1}c^{\varepsilon_1}c^{\varepsilon_2}b^{t_2} \in \langle\mathcal{H}_\Delta\rangle$.

If $\varepsilon_1 = \varepsilon_2 = 1$ then $c^2 \in \mathcal{H}_\Delta$ so $b^{t_1}c^{\varepsilon_1}c^{\varepsilon_2}b^{t_2} \in \langle\mathcal{H}_\Delta\rangle$.

If $\varepsilon_1 = \varepsilon_2 = 0$ then $b^{t_1}= b_1\cdots b_{t_1}$ corresponds to the path $r_{t_1}\cdots r_1 \in \R^{2t_1}$ and $b^{t_2}=b'_1\cdots b'_{t_2}$ corresponds to the path $r'_{t_2}\cdots r'_1 \in \R^{2t_2}$ where $b_i$ (respectively, $b'_i$) is represented by the path $r_i$ (respectively, $r'_i$) in $\R^2$ but $r'_{t_2}\cdots r'_1r_{t_1}\cdots r_1 \not\in \R^{2(t_1+t_2)}$.
The path $r'_{t_2}\cdots r'_1$ is represented by the overlap sequence
$$
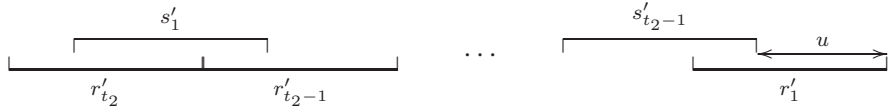
\xymatrix@W=0pt@M=0.3pt{
\ar@{^{|}-^{|}}@<-1.25ex>[rrr]_{r'_{t_2}} &
\ar@{_{|}-_{|}}@<1.25ex>[rrr]^{s'_1} & & \ar@{^{|}-^{|}}@<-1.25ex>[rrr]_{r'_{t_2-1}} & & & & \cdots & \ar@{_{|}-_{|}}@<1.25ex>[rrr]^{s'_{t_2-1}} & & \ar@{^{|}-^{|}}@<-1.25ex>[rrr]_{r'_1} &\ar@{{<}-{>}}[rr]^{u} & &
}$$
and the path $r_{t_1}\cdots r_1$ by the overlap sequence
$$
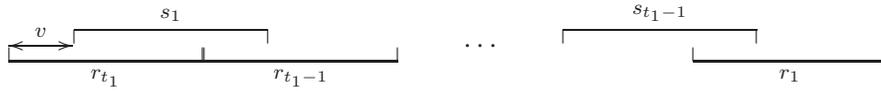
\xymatrix@W=0pt@M=0.3pt{
\ar@{^{|}-^{|}}@<-1.25ex>[rrr]_{r_{t_1}}\ar@{{<}-{>}}[r]^{v} &
\ar@{_{|}-_{|}}@<1.25ex>[rrr]^{s_1} & & \ar@{^{|}-^{|}}@<-1.25ex>[rrr]_{r_{t_1-1}} & & & & \cdots & \ar@{_{|}-_{|}}@<1.25ex>[rrr]^{s_{t_1-1}} & & \ar@{^{|}-^{|}}@<-1.25ex>[rrr]_{r_1} & & &
}$$
where $\ell(u) = D-A, \ell(v) =A$, and $s_i, s'_i \in \R^2$.
We claim that $r'_1r_{t_1} \not\in \R^4$. For otherwise, there is an overlap sequence representing $r'_1r_{t_1} \in \R^4$ of the form
$$
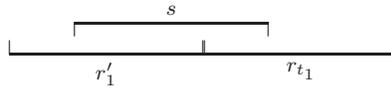
\xymatrix@W=0pt@M=0.3pt{
\ar@{^{|}-^{|}}@<-1.25ex>[rrr]_{r'_1} &
\ar@{_{|}-_{|}}@<1.25ex>[rrr]^{s} & & \ar@{^{|}-^{|}}@<-1.25ex>[rrr]_{r_{t_1}} & & &
}$$
for some $s \in R^2$. But then $$
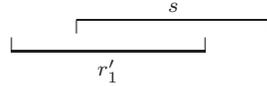
\xymatrix@W=0pt@M=0.3pt{
\ar@{^{|}-^{|}}@<-1.25ex>[rrr]_{r'_1} &
\ar@{_{|}-_{|}}@<1.25ex>[rrr]^{s} & & &
}$$
represents an element in $\R^3$ and is thus of length $D+A$. Hence
$s = uv$ as paths in $K{\mathcal Q}$, and so $r'_{t_2}\cdots r'_1r_{t_1}\cdots r_1 \in \R^{2(t_1+t_2)}$,
a contradiction. Thus, $r'_1r_{t_1} \not\in \R^4$ and we have that $b_{t_1}b'_1 = 0$
in $\Ext^4_\L(\L_0,\L_0)$. So $b_{t_1}b'_1 \in \mathcal{H}_\Delta$ and
$b^{t_1}b^{t_2}= b_1\cdots b_{t_1}b'_1\cdots b'_{t_2} \in \langle\mathcal{H}_\Delta\rangle$.

If $\varepsilon_1 = 1, \varepsilon_2 = 0$, then, from the above discussion, we may write
$b^{t_1}c$ in the form $cb^{t_1} + {\mathfrak h}$ where ${\mathfrak h} \in \langle\mathcal{H}_\Delta\rangle$, and we need to show that $cb^{t_1}b^{t_2} \in \langle\mathcal{H}_\Delta\rangle$.
Keeping the above notation, $cb^{t_1} = cb_1\cdots b_{t_1}$ corresponds to the path
$r_{t_1}\cdots r_1R^3 \in \R^{2t_1+3}$, $b^{t_2} = b'_1\cdots b'_{t_2}$ corresponds to the path $r'_{t_2}\cdots r'_1 \in \R^{2t_2}$ but
$r'_{t_2}\cdots r'_1r_{t_1}\cdots r_1R^3 \not\in \R^{2(t_1+t_2)+3}$, where $c$ is represented by $R^3\in \R^3$.
The path $r'_{t_2}\cdots r'_1$ is represented by the overlap sequence
$$
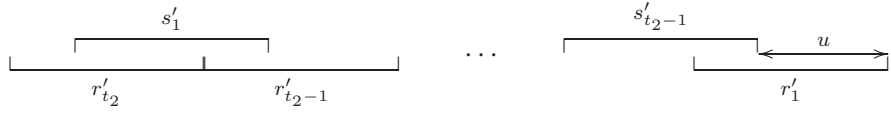
\xymatrix@W=0pt@M=0.3pt{
\ar@{^{|}-^{|}}@<-1.25ex>[rrr]_{r'_{t_2}} &
\ar@{_{|}-_{|}}@<1.25ex>[rrr]^{s'_1} & & \ar@{^{|}-^{|}}@<-1.25ex>[rrr]_{r'_{t_2-1}} & & & & \cdots & \ar@{_{|}-_{|}}@<1.25ex>[rrr]^{s'_{t_2-1}} & & \ar@{^{|}-^{|}}@<-1.25ex>[rrr]_{r'_1} &\ar@{{<}-{>}}[rr]^{u} & &
}$$
and the path $r_{t_1}\cdots r_1R^3$ by the overlap sequence
$$
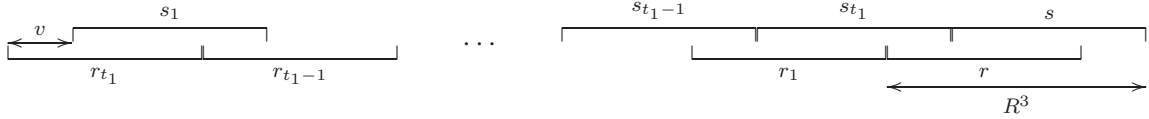
\xymatrix@W=0pt@M=0.3pt{
\ar@{^{|}-^{|}}@<-1.25ex>[rrr]_{r_{t_1}}\ar@{{<}-{>}}[r]^{v} &
\ar@{_{|}-_{|}}@<1.25ex>[rrr]^{s_1} & & \ar@{^{|}-^{|}}@<-1.25ex>[rrr]_{r_{t_1-1}} & & & & \cdots & \ar@{_{|}-_{|}}@<1.25ex>[rrr]^{s_{t_1-1}} & & \ar@{^{|}-^{|}}@<-1.25ex>[rrr]_{r_1} & \ar@{_{|}-_{|}}@<1.25ex>[rrr]^{s_{t_1}} & &
\ar@{^{|}-^{|}}@<-1.25ex>[rrr]_{r}\ar@{{<}-{>}}@<-3.5ex>[rrrr]_{R^3} & \ar@{_{|}-_{|}}@<1.25ex>[rrr]^{s} & & &
}$$
where $\ell(u) = D-A, \ell(v) =A$, and $r, s, s_i, s'_i \in \R^2$.
Using the same arguments as in the case $\varepsilon_1 = \varepsilon_2 = 0$,
we see that $r'_1r_{t_1} \not\in \R^4$. Thus $b_{t_1}b'_1 \in \mathcal{H}_\Delta$ and
$cb^{t_1}b^{t_2}= cb_1\cdots b_{t_1}b'_1\cdots b'_{t_2} \in \langle\mathcal{H}_\Delta\rangle$.

If $\varepsilon_1 = 0, \varepsilon_2 = 1$ then, from the above discussion, we may write
$cb^{t_2}$ in the form $b^{t_2}c + {\mathfrak h}$ where ${\mathfrak h} \in \langle\mathcal{H}_\Delta\rangle$, and we need to show that $b^{t_1}b^{t_2}c \in \langle\mathcal{H}_\Delta\rangle$.
Keeping the above notation, $b^{t_1} = b_1\cdots b_{t_1}$ corresponds to the path
$r_{t_1}\cdots r_1 \in \R^{2t_1}$, $b^{t_2}c = b'_1\cdots b'_{t_2}c$ corresponds to the path $R^3r'_{t_2}\cdots r'_1 \in \R^{2t_2+3}$ but
$R^3r'_{t_2}\cdots r'_1r_{t_1}\cdots r_1 \not\in \R^{2(t_1+t_2)+3}$, where $c$ is represented by $R^3\in \R^3$.
The path $R^3r'_{t_2}\cdots r'_1$ is represented by the overlap sequence
$$
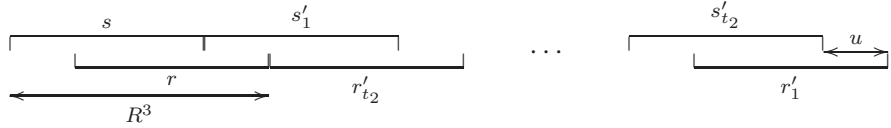
\xymatrix@W=0pt@M=0.3pt{
\ar@{_{|}-_{|}}@<1.25ex>[rrr]^{s}\ar@{{<}-{>}}@<-3.5ex>[rrrr]_{R^3} &
\ar@{^{|}-^{|}}@<-1.25ex>[rrr]_{r} & &
\ar@{_{|}-_{|}}@<1.25ex>[rrr]^{s'_1} & \ar@{^{|}-^{|}}@<-1.25ex>[rrr]_{r'_{t_2}} & & & & \cdots & \ar@{_{|}-_{|}}@<1.25ex>[rrr]^{s'_{t_2}} & \ar@{^{|}-^{|}}@<-1.25ex>[rrr]_{r'_1} & & \ar@{{<}-{>}}[r]^{u} &
}$$
and the path $r_{t_1}\cdots r_1$ by the overlap sequence
$$
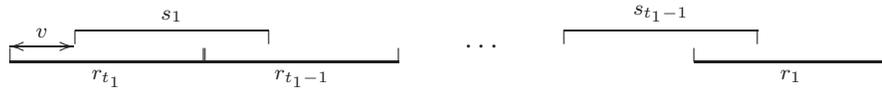
\xymatrix@W=0pt@M=0.3pt{
\ar@{^{|}-^{|}}@<-1.25ex>[rrr]_{r_{t_1}}\ar@{{<}-{>}}[r]^{v} &
\ar@{_{|}-_{|}}@<1.25ex>[rrr]^{s_1} & & \ar@{^{|}-^{|}}@<-1.25ex>[rrr]_{r_{t_1-1}} & & & & \cdots & \ar@{_{|}-_{|}}@<1.25ex>[rrr]^{s_{t_1-1}} & & \ar@{^{|}-^{|}}@<-1.25ex>[rrr]_{r_1} & & &
}$$
where $\ell(u) = \ell(v) =A$, and $r, s, s_i, s'_i \in \R^2$.

We claim that $r'_1r_{t_1} \not\in \R^4$. For otherwise, there is an overlap sequence representing $r'_1r_{t_1} \in \R^4$ of the form
$$
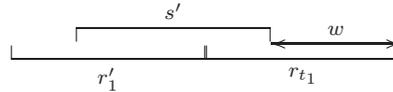
\xymatrix@W=0pt@M=0.3pt{
\ar@{^{|}-^{|}}@<-1.25ex>[rrr]_{r'_1} &
\ar@{_{|}-_{|}}@<1.25ex>[rrr]^{s'} & & \ar@{^{|}-^{|}}@<-1.25ex>[rrr]_{r_{t_1}} & \ar@{{<}-{>}}[rr]^{w} & &
}$$
for some $s' \in R^2$ and with $r_{t_1} = vw$ as paths.
But then $r_{t_1}$ overlaps $s'$ with overlap $s'w$ of length $2D-A$.
By assumption, $D \neq 2A$, so, as every element of $R^3$ is of length $D+A$, this
is not a maximal overlap. Thus there is $\tilde{s} \in \R^3$ and a maximal overlap of $r_{t_1}$ with $\tilde{s}$ as follows:
$$
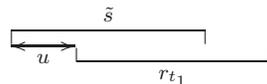
\xymatrix@W=0pt@M=0.3pt{
\ar@{_{|}-_{|}}@<1.25ex>[rrr]^{\tilde{s}}\ar@{{<}-{>}}[r]_{u} &
\ar@{^{|}-^{|}}@<-1.25ex>[rrr]_{r_{t_1}} & & &
}$$
Hence, the path $R^3r'_{t_2}\cdots r'_1r_{t_1}$ is in $\R^{2(t_2+1)+3}$ and is represented by the maximal overlap sequence
$$
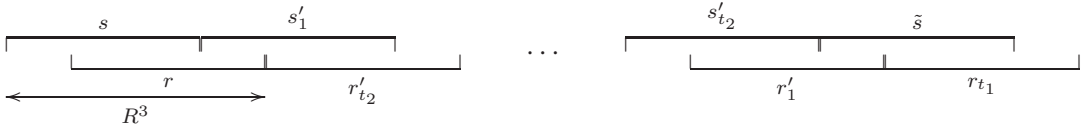
\xymatrix@W=0pt@M=0.3pt{
\ar@{_{|}-_{|}}@<1.25ex>[rrr]^{s}\ar@{{<}-{>}}@<-3.5ex>[rrrr]_{R^3} &
\ar@{^{|}-^{|}}@<-1.25ex>[rrr]_{r} & &
\ar@{_{|}-_{|}}@<1.25ex>[rrr]^{s'_1} & \ar@{^{|}-^{|}}@<-1.25ex>[rrr]_{r'_{t_2}} & & & & \cdots & \ar@{_{|}-_{|}}@<1.25ex>[rrr]^{s'_{t_2}} & \ar@{^{|}-^{|}}@<-1.25ex>[rrr]_{r'_1} & & \ar@{_{|}-_{|}}@<1.25ex>[rrr]^{\tilde{s}} & \ar@{^{|}-^{|}}@<-1.25ex>[rrr]_{r_{t_1}} & & &
}$$
Now, we may continue inductively and show that $R^3r'_{t_2}\cdots r'_1r_{t_1}\cdots r_1$ is in $\R^{2(t_2+t_1)+3}$, which is a contradiction.
Hence $r'_1r_{t_1} \not\in \R^4$. Thus $b_{t_1}b'_1 \in \mathcal{H}_\Delta$ and
$b^{t_1}b^{t_2}c = b_1\cdots b_{t_1}b'_1\cdots b'_{t_2}c \in \langle\mathcal{H}_\Delta\rangle$.

This shows that, in all cases, $f^{n_1}_{i_1}f^{n_2}_{i_2} \in \langle\mathcal{H}_\Delta\rangle$.

Now suppose that $f^{n_1}_{i_1} \cdots f^{n_r}_{i_r} - f^{m_1}_{i_1} \cdots f^{m_s}_{i_s}$ is in $I_\Delta$.
Then, as paths in
$K{\mathcal Q}$, we have $R^{n_r}_{i_r} \cdots R^{n_1}_{i_1} = R^{m_s}_{i_s} \cdots R^{m_1}_{i_1}$
and $n_1 + \cdots + n_r = m_1 + \cdots + m_s$. Set $n = n_1 + \cdots + n_r$.
Clearly we may assume that neither $f^{n_1}_{i_1} \cdots f^{n_r}_{i_r}$ nor $f^{m_1}_{i_1} \cdots f^{m_s}_{i_s}$ is in $I_\Delta$.
So, we may write $f^{n_1}_{i_1} \cdots f^{n_r}_{i_r} = b^{t_1}c^{\varepsilon_1} + {\mathfrak h_1}$ and
$f^{m_1}_{i_1} \cdots f^{m_s}_{i_s} = b^{t_2}c^{\varepsilon_2} + {\mathfrak h_2}$,
where $\varepsilon_1, \varepsilon_2 \in \{0, 1\}$ and ${\mathfrak h}_1, {\mathfrak h}_2 \in \langle\mathcal{H}_\Delta\rangle$.
However, as paths in $K{\mathcal Q}$, the element of $\R^n$ corresponding to $b^{t_1}c^{\varepsilon_1}$
is equal to $R^{n_r}_{i_r} \cdots R^{n_1}_{i_1}$, and the element of $\R^n$ corresponding to $b^{t_2}c^{\varepsilon_2}$ is equal to $R^{m_s}_{i_s} \cdots R^{m_1}_{i_1}$.
Since these paths are equal, it follows that $b^{t_1}c^{\varepsilon_1} = b^{t_2}c^{\varepsilon_2}$.
Thus $f^{n_1}_{i_1} \cdots f^{n_r}_{i_r} - f^{m_1}_{i_1} \cdots f^{m_s}_{i_s} =
{\mathfrak h_1} - {\mathfrak h_2} \in \langle\mathcal{H}_\Delta\rangle$.

Thus all the generators of $I_\Delta$ lie in $\langle\mathcal{H}_\Delta\rangle$, and we have shown
that $I_\Delta = \langle\mathcal{H}_\Delta\rangle$.
It is now clear that $\mathcal{H}_\Delta$ is a minimal generating set for $I_\Delta$.
\end{proof}

To show that $\mathcal{H}_\Delta$ is a reduced Gr\"{o}bner basis for $I_\Delta$, we need an admissible order. Let $\mathcal{B}$ be the basis of $K\Delta$ which consists of all paths.
Label the elements of each set $f^0, f^1, f^2$ and $f^3$ as follows:
$$\begin{array}{l}
f^0  =  \{e_1, e_2, \dots, e_m\}\\
f^1  =  \{a_1, a_2, \dots, a_r\}\\
f^2  =  \{b_1, b_2, \dots, b_s\}\\
f^3  =  \{c_1, c_2, \dots, c_t\}.
\end{array}$$
Order the vertices and arrows of
$\Delta$ by
$$a_1 > a_2 > \cdots > a_r > b_1 > b_2 > \cdots > b_s > c_1 > c_2 > \cdots > c_t > e_1 > e_2 > \cdots > e_m.$$
Let $>$ be the left length lexicographic order on $\mathcal{B}$ as given in Section~\ref{sec:GB}, so that $>$ is an admissible order.

\begin{prop}\label{prop:monredGB}
Let $\Lambda = K{\mathcal Q}/I$ be a $(D,A)$-stacked monomial algebra with $D>2, A> 1$ and $D \neq 2A$. Keeping the above notation, the set $\mathcal{H}_\Delta$ is a quadratic reduced Gr\"{o}bner basis of $I_\Delta$.
\end{prop}

\begin{proof}
For ease of notation within the proof, write $\mathcal{H}$ for $\mathcal{H}_\Delta$. From Proposition~\ref{prop:monredgen}, the set $\mathcal{H}$ generates the ideal $I_\Delta$, and it clearly consists of uniform quadratic elements of $K\Delta$. We apply Theorem~\ref{reducedG} to show that $\mathcal{H}$ is a reduced Gr\"{o}bner basis of $I_\Delta$. By inspection, $\CTip(h) = 1$ for all $h \in {\mathcal H}$ so (i) holds.

To show (ii), let $h, h'$ be distinct elements of $\mathcal{H}$. Since both $h$ and $h'$ have length 2, if $h$ reduces over $h'$ then $\Tip(h') \in \Supp(h)$. But, by inspection of $\mathcal{H}$ we see this is never the case. Thus $h$ does not reduce over $h'$, and (ii) holds.

For (iii), we consider the overlap difference for two (not necessarily distinct) members of
$\mathcal{H}$. By inspection of $\mathcal{H}$, it is clear that in order to have an
overlap difference of elements $h, h' \in \mathcal{H}$, then at least one of $h, h'$ is a monomial.
Moreover, we note that the overlap difference of two monomial elements always reduces to zero. Thus we only need to consider the overlap difference of one monomial element and one non-monomial element.

We have the following six cases.
First, suppose $h = bc-c'b'$. If $h' = ca$, then $o(h, h', a, b) = -c'b'a \Rightarrow_{\mathcal H} 0$ since $b'a \in {\mathcal H}$. For $h' = c\tilde{b}$ we have $o(h, h', \tilde{b}, b) = -c'b'\tilde{b}$. We show that $b'\tilde{b} \in {\mathcal H}$.

Let $r, r', \tilde{r}$ in $\R^2$ represent $b, b', \tilde{b}$ respectively, and let $R, R'$ in $\R^3$ represent $c, c'$ respectively.
Then, $bc$ corresponds to the path $Rr \in \R^5$, $c'b'$ corresponds to the path $r'R' \in \R^5$ with
$Rr = r'R'$ as paths in $K{\mathcal Q}$, and $\tilde{r}R \not\in \R^5$. We can represent $Rr$ and $r'R'$ by the overlap sequence
$$
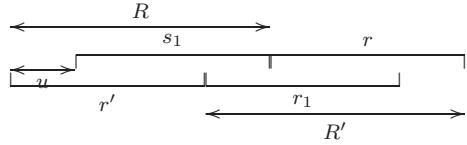
\xymatrix@W=0pt@M=0.3pt{
\ar@{^{|}-^{|}}@<-1.25ex>[rrr]_{r'}\ar@{{<}-{>}}[r]_{u}\ar@{{<}-{>}}@<3.5ex>[rrrr]^{R} & \ar@{_{|}-_{|}}@<1.25ex>[rrr]^{s_1} & &
\ar@{^{|}-^{|}}@<-1.25ex>[rrr]_{r_1}\ar@{{<}-{>}}@<-3.5ex>[rrrr]_{R'} & \ar@{_{|}-_{|}}@<1.25ex>[rrr]^{r} & & &
}$$
with $r_1, s_1 \in \R^2$ and $\ell(u) = A$.
If $b'\tilde{b} \not\in {\mathcal H}$ so that $\tilde{r}r' \in \R^4$, then $\tilde{r}r'$ is represented by the maximal overlap sequence
$$
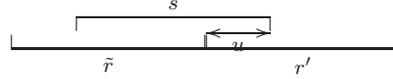
\xymatrix@W=0pt@M=0.3pt{
\ar@{^{|}-^{|}}@<-1.25ex>[rrr]_{\tilde{r}} & \ar@{_{|}-_{|}}@<1.25ex>[rrr]^{s} & &
\ar@{^{|}-^{|}}@<-1.25ex>[rrr]_{r'}\ar@{{<}-{>}}[r]_{u} & & &
}$$
for some $s \in \R^2$. Then it is immediate that
$\tilde{r}R$ is represented by the maximal overlap sequence
$$
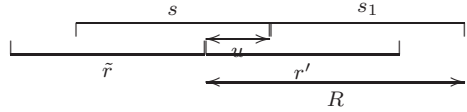
\xymatrix@W=0pt@M=0.3pt{
\ar@{^{|}-^{|}}@<-1.25ex>[rrr]_{\tilde{r}} & \ar@{_{|}-_{|}}@<1.25ex>[rrr]^{s} & &
\ar@{^{|}-^{|}}@<-1.25ex>[rrr]_{r'}\ar@{{<}-{>}}[r]_{u}\ar@{{<}-{>}}@<-3.5ex>[rrrr]_{R} & \ar@{_{|}-_{|}}@<1.25ex>[rrr]^{s_1} & & &
}$$
and so $\tilde{r}R \in \R^5$; this is a contradiction.
Hence $b'\tilde{b} \in {\mathcal H}$ and
$o(h, h', \tilde{b}, b) = -c'b'\tilde{b} \Rightarrow_{\mathcal H} 0$.

For $h' = c\tilde{c}$ we have $o(h, h', \tilde{c}, b) = -c'b'\tilde{c}$.
If $b'\tilde{c} \in {\mathcal H}$ then $-c'b'\tilde{c} \Rightarrow_{\mathcal H} 0$.
Otherwise, if $b'\tilde{c} \not\in {\mathcal H}$ then, from Proposition~\ref{prop:R5}(ii),
we have that $b'\tilde{c} - \hat{c}\hat{b} \in {\mathcal H}$ for some $\hat{b}\in f^2, \hat{c} \in f^3$.
Thus $-c'b'\tilde{c} \Rightarrow_{\mathcal H} -c'\hat{c}\hat{b}$.
But $c'\hat{c} \in {\mathcal H}$ and thus $-c'b'\tilde{c} \Rightarrow_{\mathcal H} 0$.
Hence, on both occasions, $o(h, h', \tilde{c}, b) \Rightarrow_{\mathcal H} 0$.

The remaining cases are where $h' = bc-c'b'$ and $h \in \{ab, \tilde{b}b, \tilde{c}b\}$; these are similar to the above and are left to the reader.
Thus (iii) holds and $\mathcal{H}_\Delta$ is a quadratic reduced Gr\"{o}bner basis of $I_\Delta$.
\end{proof}

The next result now follows immediately from Theorem~\ref{Kosz}(ii).

\begin{thm}\label{thm:main_monomial_thm}
Let $\L= K{\mathcal Q}/I$ be a $(D,A)$-stacked monomial algebra with $D > 2, A > 1$ and $D \neq 2A$. Let $\hat{E}(\L)$ be the Ext algebra of $\L$ with the hat-degree grading. Then $\hat{E}(\L)$ is a Koszul algebra.
\end{thm}

Following the discussion at the beginning of this section and using~\cite[Theorem 7.1]{bib:GMMZ}, we have the following theorem.

\begin{thm}\label{thm:main_monomial_thm_2}
Let $\L= K{\mathcal Q}/I$ be a $(D,A)$-stacked monomial algebra with $D \geqslant 2, A \geqslant 1$. Assume that $D \neq 2A$ when $A > 1$. Then there is a regrading of the Ext algebra of $\L$ so that the regraded Ext algebra is a Koszul algebra.
\end{thm}

We conclude this section by showing the necessity of the condition $\gldim \L \geqslant 6$ in Proposition~\ref{NOregrading}.

\begin{example}\label{example gldim 6}
Let $\L = K\cQ/I$ where $\cQ$ is the quiver
$$\xymatrix{
1\ar[r]^{\a_1} & 2\ar[r]^{\a_2} & 3\ar[r] & \cdots\ar[r] & 10\ar[r]^{\a_{10}} & 11
}$$
and $I = \langle \a_1 \a_2 \a_3 \a_4, \a_3 \a_4 \a_5 \a_6, \a_5 \a_6\a_7 \a_8, \a_7 \a_8 \a_9 \a_{10} \rangle$.
Then $\L$ is a monomial algebra, so using \cite{bib:GS2}, we have that $\L$ is a $(4, 2)$-stacked
monomial algebra. Moreover, $\L$ has global dimension $5$. The set $\R^2$ is the minimal generating set for $I$ above, and the set $\R^3$ is
$$\R^3 = \{\a_1 \a_2 \a_3 \a_4\a_5 \a_6,\ \a_3 \a_4 \a_5 \a_6\a_7 \a_8,\ \a_5 \a_6\a_7 \a_8\a_9 \a_{10}\}.$$
Let
$$\begin{array}{rcl}
\hat{E}(\L)_0 & = & \Ext^0_{\Lambda} (\Lambda_0 , \Lambda_0)\\
\hat{E}(\L)_1 & = & \Ext^1_{\Lambda} (\Lambda_0 , \Lambda_0)\p \Ext^2_{\Lambda} (\Lambda_0 , \Lambda_0) \p \Ext^3_{\Lambda} (\Lambda_0 , \Lambda_0) \\
\hat{E}(\L)_2 & = & \Ext^{4}_{\Lambda} (\Lambda_0 , \Lambda_0) \p \Ext^{5}_{\Lambda} (\Lambda_0 , \Lambda_0).
\end{array}$$
Then the Ext algebra of $\Lambda$ is $\hat{E}(\L) = \hat{E}(\L)_0\p\hat{E}(\L)_1\p\hat{E}(\L)_2.$
Using Theorem~\ref{summary} and that $\Ext^6_{\L} (\L_0, \L_0)=0$, it follows that $\hat{E}_1(\L)\times\hat{E}_1(\L) = \hat{E}_2(\L)$. Thus the Ext algebra is a graded algebra with the hat-degree.

It may be easily verified that $\hat{E}(\L) = K\Delta/I_\Delta$ where $\Delta$ is the quiver
$$
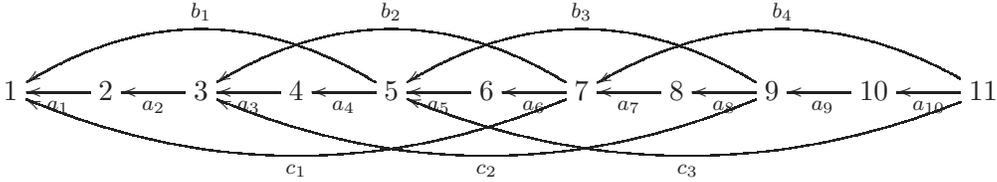
\xymatrix{
1 & 2\ar[l]^{a_1} &
3\ar[l]^{a_2} & 4\ar[l]^{a_3} &
5\ar[l]^{a_4}\ar@/_2pc/[llll]_{b_1} & 6\ar[l]^{a_5} &
7\ar[l]^{a_6}\ar@/_2pc/[llll]_{b_2}\ar@/^2pc/[llllll]^{c_1} &
8\ar[l]^{a_7} & 9\ar[l]^{a_8}\ar@/_2pc/[llll]_{b_3}\ar@/^2pc/[llllll]^{c_2} & 10\ar[l]^{a_9} & 11\ar[l]^{a_{10}}\ar@/_2pc/[llll]_{b_4}\ar@/^2pc/[llllll]^{c_3}
}$$
and $I_\Delta$ is generated by:
$$\begin{array}{l}
a_ia_{i-1} \mbox { for $i = 2, \dots, 10$}\\
a_5b_1, a_7b_2, a_9b_3, b_2a_2, b_3a_4, b_4a_6, a_7c_1, a_9c_2, c_2a_2, c_3a_4\\
b_4c_1 - c_3b_1
\end{array}$$
where the $a_i, b_i, c_i$ correspond to elements in $\Ext^1_{\L} (\L_0, \L_0), \Ext^2_{\L} (\L_0, \L_0), \Ext^3_{\L} (\L_0, \L_0)$ respectively.
Let ${\mathcal H}$ be the set of generators for $I_\Delta$ given above. Order the
vertices and arrows of $\Delta$ by
$$a_1 > \cdots > a_{10} > b_1 > \cdots > b_4 > c_1 > c_2 > c_3 > e_1 > \cdots > e_{11}.$$
Let $>$ be the left length lexicographic order on $\mathcal{B}$ as given in Section~\ref{sec:GB}, so that $>$ is an admissible order.
There are no overlaps of $b_4c_1 - c_3b_1$ with any element of ${\mathcal H}$; so the only overlaps of elements of ${\mathcal H}$ are of monomials in ${\mathcal H}$, and these immediately reduce to zero. It follows from Theorem~\ref{reducedG} that ${\mathcal H}$ is a reduced Gr\"obner basis for $I_\Delta$. From Theorem~\ref{Kosz}, we then have that $\hat{E}(\L)$ is a Koszul algebra with this regrading.
\end{example}

In the final section we consider the non-monomial case.

\section{An example of a non-monomial $(D,A)$-stacked algebra}\label{sec:notmonomial}

We conclude this paper with an example of a $(D,A)$-stacked algebra which is not monomial
but where the regraded Ext algebra is still Koszul. Our example has $D \neq 2A$ and is of infinite global dimension.

\begin{example}\label{nonfish}
Let $\cQ$ be the quiver given by
$$
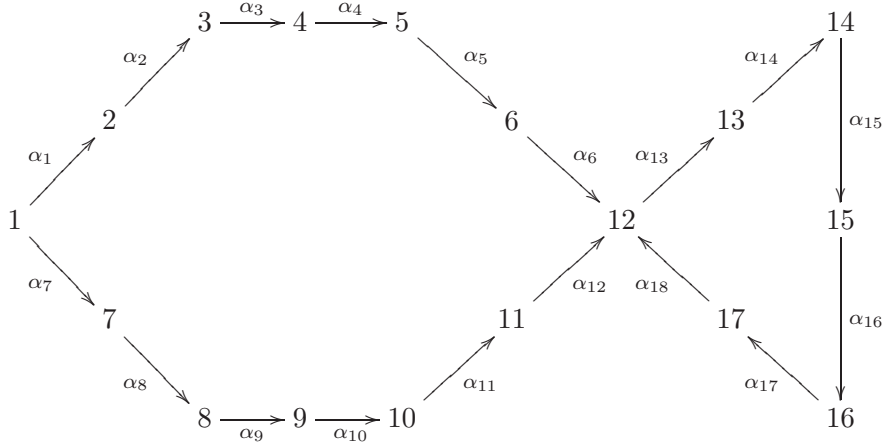
\xymatrix{
& & 3\ar[r]^{\a_3} & 4\ar[r]^{\a_4} & 5\ar[dr]^{\a_5} & & & & 14\ar[dd]^{\a_{15}}\\
& 2\ar[ur]^{\a_2} & & & & 6\ar[dr]^{\a_6} & & 13\ar[ur]^{\a_{14}} & \\
1\ar[ur]^{\a_1}\ar[dr]_{\a_7} & & & & & & 12\ar[ur]^{\a_{13}} & & 15\ar[dd]^{\a_{16}}\\
& 7\ar[dr]_{\a_8} & & & & 11\ar[ur]_{\a_{12}} & & 17\ar[ul]^{\a_{18}} & \\
& & 8\ar[r]_{\a_9}& 9\ar[r]_{\a_{10}} & 10\ar[ur]_{\a_{11}} & & & & 16\ar[ul]^{\a_{17}}
}$$
and let $\L = K\cQ / I$ where $I$ is generated by the elements
$$g^2_1 = \a_1 \a_2 \a_3 \a_4 \a_5 \a_6 - \a_7 \a_8 \a_9 \a_{10} \a_{11} \a_{12},\ g^2_2 = \a_3 \a_4 \a_5 \a_6 \a_{13} \a_{14},$$
$$g^2_3 = \a_5 \a_6 \a_{13} \a_{14} \a_{15} \a_{16}, \ g^2_4 = \a_9 \a_{10} \a_{11} \a_{12} \a_{13} \a_{14}, \ g^2_5 = \a_{11} \a_{12} \a_{13} \a_{14} \a_{15} \a_{16},$$
$$g^2_6 = \a_{13} \a_{14} \a_{15} \a_{16} \a_{17} \a_{18}, \ g^2_7 = \a_{15} \a_{16} \a_{17} \a_{18} \a_{13} \a_{14}, \ g^2_8 = \a_{17} \a_{18} \a_{13} \a_{14} \a_{15} \a_{16}.$$

The first step is to show that $\L$ is a $(6,2)$-stacked algebra. We start by finding a minimal projective resolution of $\L_0$ using the approach of \cite{bib:GSZ}. We define sets $g^n$ in $K\cQ$ for each $n \geqslant 0$ inductively as follows:
\begin{enumerate}
\item[$\bullet$] $g^0 = \{ g^0_i = e_i \mbox{ for } i = 1, \dots , 17 \}$, the set of vertices of $\cQ$.
\item[$\bullet$] $g^1 = \{ g^1_i = \a_i \mbox{ for } i = 1, \dots , 18 \}$, the set of arrows of $\cQ$.
\item[$\bullet$] $g^2 = \{ g^2_1, \dots , g^2_8\}$, the minimal generating set for $I$ as described above.
\item[$\bullet$] For $n=2r+1$ with $r\geqslant 1$,\\
$g^n = \{ g^n_1 = g^{n-1}_1\a_{13} \a_{14}, \ g^n_2 = g^{n-1}_2\a_{15} \a_{16}, \ g^n_3 = g^{n-1}_3\a_{17}\a_{18}, \ g^n_4 = g^{n-1}_4 \a_{15}\a_{16}, \\
\hspace*{1.5cm} g^n_5 = g^{n-1}_5\a_{17}\a_{18}, \ g^n_6 = g^{n-1}_6\a_{13}\a_{14} , \ g^n_7 = g^{n-1}_7\a_{15} \a_{16}, \ g^n_8 = g^{n-1}_8\a_{17} \a_{18}\}.$
\item[$\bullet$] For $n=2r$ with $r\geqslant 2$,\\
$ g^n = \{ g^n_1 = g^{n-1}_1\a_{15} \a_{16} \a_{17} \a_{18}, \ g^n_2 = g^{n-1}_2\a_{17} \a_{18} \a_{13} \a_{14}, \ g^n_3 = g^{n-1}_3\a_{13} \a_{14} \a_{15} \a_{16}, \\
\hspace*{1.5cm} g^n_4 = g^{n-1}_4 \a_{17} \a_{18} \a_{13} \a_{14}, \ g^n_5 = g^{n-1}_5\a_{13} \a_{14} \a_{15} \a_{16}, \ g^n_6 = g^{n-1}_6\a_{15} \a_{16} \a_{17} \a_{18}, \\
\hspace*{1.5cm} g^n_7 = g^{n-1}_7 \a_{17} \a_{18} \a_{13} \a_{14}, \ g^n_8 =  g^{n-1}_1\a_{13} \a_{14} \a_{15} \a_{16}\}$.
\end{enumerate}
Note that we may write $g^2_1 = g^1_1 \a_2 \a_3 \a_4 \a_5 \a_6 - g^1_7 \a_8 \a_9 \a_{10} \a_{11} \a_{12}$. It follows that each $x \in g^n$ may be written as $x=\sum_i g^{n - 1}_ir_i$
for some $r_i \in K\cQ$, and all $n \geqslant 1$. Let $P^n$ be the projective $\L$-module defined by
$P^n = \bigoplus_i t(g^n_i)\L$. Let $d^0: P^0 \to \L/\rrad$ be the canonical surjection and,
for $n \geqslant 1$, define $\L$-homomorphisms $d^n: P^n \to P^{n-1}$ by
$t(g^n_i) \mapsto t(g^{n - 1}_i)r_i$ where $t(g^{n - 1}_i)r_i$ is in the component of
$P^{n-1}$ corresponding to $t(g^{n-1}_i)$.
It may be checked that $(P^n, d^n)$ is a minimal projective resolution for $\L / \rrad$ as a
right $\L$-module. By considering the length of each $g^n$, it follows from Definition~\ref{d,a}
that $\L$ is a $(6,2)$-stacked algebra.

We identify each element in the set $g^n$ with a basis element in $\Ext^n_\L(\L_0, \L_0)$ in
the following way.
Let $n \geqslant 0$ and define $f^n_i$ to be the $\L$-homomorphism $P^n \to \L / \rrad$ given by
$$f^n_i : t(g^n_j) \mapsto
\left \{ \begin{array}{ll}
t(g^n_i) + \rrad & \mbox{ if }  i = j \\
0 & \mbox{ otherwise}.
\end{array} \right. $$
We set $f^n = \{ f^n_i \}$ so that $f^n$ is a basis of $\Ext^n_\L(\L_0, \L_0)$ for each $n \geqslant 0$.
It is now a straightforward exercise to compute the products in $E(\L)$. For example,
if $n = 2r$ with $r \geqslant 2$, then
$f^n_1 = f^2_6 \cdot f^{n-2}_1$, since $f^n_1 = f^2_6 \circ \mathcal{L}^2 f^{n-2}_1$ as maps, where the lifting $\mathcal{L}^2 f^{n-2}_1 : P^n \to P^2$ can be chosen as
$$\mathcal{L}^2 f^{n-2}_1 : t(g^n_j) \mapsto
\left \{ \begin{array}{ll}
t(g^2_6) & \mbox{ if }  j=1 \\
0 & \mbox{ otherwise}.
\end{array} \right. $$
Since $E(\L)$ is generated in degrees $0, 1, 2$ and $3$, we have that $f^0\cup f^1\cup f^2\cup f^3$ is a minimal generating set for $E(\L)$. Set $\a_i = f^1_i$, $\b_i = f^2_i$ and $\g_i = f^3_i$ for all $i$.
Then, $E(\L)$ is given by quiver and relations as $K\Gamma/I_\Gamma$, where $\Gamma$ is the quiver
$$
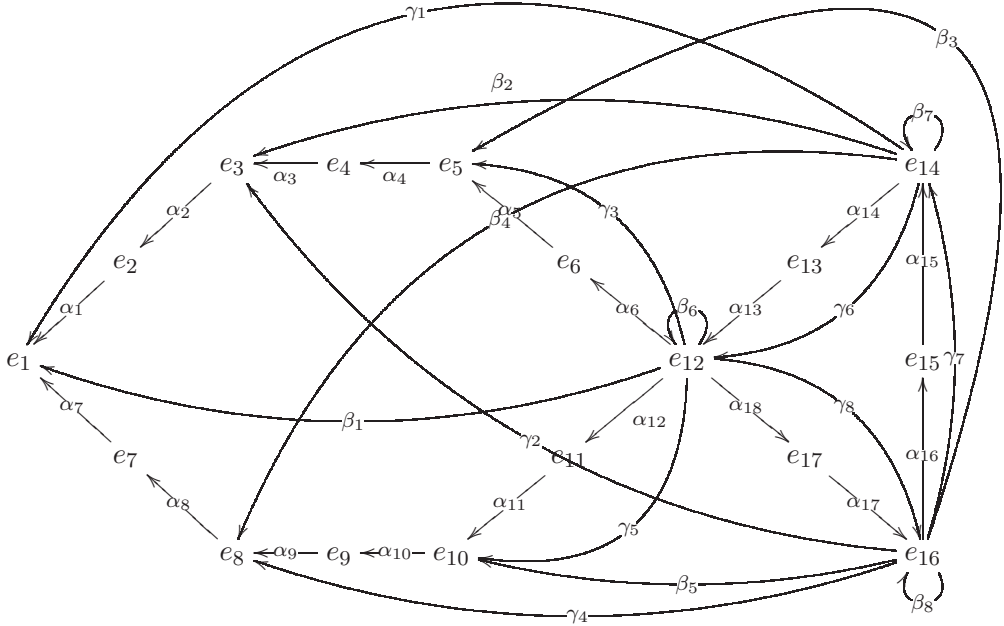
\xymatrix{
& & e_3 \ar[dl]|{\a_{2}} & e_4 \ar[l]^{\a_{3}} & e_5\ar[l]^{\a_{4}} & & & & e_{14}\ar@/^2pc/[ddll]|{\gamma_6}\ar@/_8pc/[ddllllllll]|{\gamma_1}\ar@(ur,ul)[]|{\b_7} \ar@/_5pc/[ddddllllll]|{\b_4}\ar@/_2pc/[llllll]_{\hspace{-2cm}\b_2} \ar[dl]|{\a_{14}} \\
&  e_2\ar[dl]|{\a_{1}} & & & & e_6 \ar[ul]|{\a_{5}}& &  e_{13}\ar[dl]|{\a_{13}} & \\
e_1  & & & & & & e_{12}\ar@/^2pc/[ddrr]|{\gamma_8}\ar@/^3pc/[ddll]|{\gamma_5}\ar@[blue]@/_2pc/[uull]|{\gamma_3} \ar@(ur,ul)[]|{\b_6}\ar[ul]|{\a_{6}}\ar@/^2pc/[llllll]|{\b_1} \ar[dr]|{\a_{18}}\ar[dl]^{\a_{12}} & & e_{15} \ar[uu]|{\a_{15}} \\
& e_7\ar[ul]|{\a_{7}} & & & & e_{11} \ar[dl]|{\a_{11}} & & e_{17}\ar[dr]|{\a_{17}}  & \\
& & e_8\ar[ul]|{\a_{8}} & e_9\ar[l]|{\a_{9}} & e_{10}\ar[l]|{\a_{10}}  & & & & e_{16}\ar@/_1pc/[uuuu]|{\gamma_7}\ar@/^2pc/[llllll]|{\gamma_4}\ar@/^3pc/[uuuullllll]|{\gamma_2} \ar@(dr,dl)[]|{\b_8}\ar@/^1pc/[llll]|{\b_5}\ar@/_13pc/[uuuullll]|{\b_3} \ar[uu]|{\a_{16}}\\
}$$
and $I_\Gamma$ is the ideal generated by the set $\mathcal{H}$, whose elements are listed as follows:
\begin{enumerate}
\item[$\bullet$] $ \a_2 \a_1, \a_3 \a_2, \a_4 \a_3, \a_5 \a_4, \a_6 \a_5, \a_8 \a_7, \a_9 \a_8, \a_{10} \a_9, \a_{11} \a_{10}, \a_{12} \a_{11}, \a_{13} \a_{6}, \a_{13} \a_{12},\\ \a_{13} \a_{18}, \a_{13} \b_{1}, \a_{13} \b_6, \a_{13} \g_{3}, \a_{13} \g_{5}, \a_{13} \g_{8}, \a_{14} \a_{13}, \a_{15} \a_{14}, \a_{15} \b_2, \a_{15} \b_4, \a_{15} \b_7,\\ \a_{15} \g_1, \a_{15} \g_{6}, \a_{16} \a_{15}, \a_{17} \a_{16}, \a_{17} \b_3, \a_{17} \b_5, \a_{17} \b_8, \a_{17} \g_{2}, \a_{17} \g_{4}, \a_{17} \g_{7}, \a_{18} \a_{17}, $
\item[$\bullet$] $ \b_2 \a_{2}, \b_3 \a_{4}, \b_4 \a_{8}, \b_5 \a_{10}, \b_6 \a_{6}, \b_6 \a_{12}, \b_6 \a_{18}, \b_6 \gamma_3 - \gamma_8 \b_3 , \b_6 \gamma_5 - \gamma_8 \b_5 , \b_6 \gamma_8 - \gamma_8 \b_8,\\  \b_7 \a_{14},  \b_7 \gamma_1 - \gamma_6 \b_1, \b_7 \gamma_6 - \gamma_6 \b_6, \b_8 \a_{16}, \b_8 \gamma_2 -\gamma_7 \b_2, \b_8 \gamma_4 - \gamma_7 \b_4, \b_8 \gamma_7 - \gamma_7 \b_7, $
\item[$\bullet$] $ \g_2 \a_{2}, \g_3 \a_{4}, \g_4 \a_{8}, \g_5 \a_{10},  \g_6 \a_{6}, \g_6 \a_{12}, \g_6 \a_{18}, \g_6 \g_3, \g_6 \g_5, \g_6 \g_8,  \g_7 \a_{14}, \g_7 \g_1, \g_7 \g_6, \\ \g_8 \a_{16}, \g_8 \g_2, \g_8 \g_4, \g_8 \g_7$.
\end{enumerate}
In particular, $\mathcal{H}$ is a minimal generating set for $I_\Gamma$. This gives the following result.

\begin{prop}\label{A=E}
Let $\L$ and $K\Gamma/I_\Gamma$ be as above. Let $\hat{E}(\L)$ be the Ext algebra of $\L$ with the hat-degree grading. Then $\hat{E}(\L)\cong K\Gamma/I_\Gamma$ where $\a_i, \b_i$ and $\g_i$ are in degree $1$, corresponding to the basis elements of $\hat{E}(\L)_1$.
\end{prop}

We now show that $K\Gamma/I_\Gamma$ is a Koszul algebra by showing that $\mathcal{H}$
is a quadratic reduced Gr\"{o}bner basis for the ideal $I_\Gamma$. Let $\mathcal{B}$ be the basis of $K\Gamma$ which consists of all paths. We order the vertices and arrows of $\Gamma$ by
$$\a_1 > \a_2 > \cdots > \a_{18} > \b_1 > \b_2 > \cdots > \b_8 > \g_1 > \g_2 > \cdots > \g_8 > e_1 > e_2 > \cdots > e_{17}$$
and let $>$ be the left length lexicographic order on $\mathcal{B}$ as given in Section~\ref{sec:GB}. Then $>$ is an admissible order.

\begin{prop}
With the above notation, $\mathcal{H}$ is a quadratic reduced Gr\"{o}bner basis of $I_\Gamma$.
\end{prop}

\begin{proof}
We apply Theorem~\ref{reducedG}. The set $\mathcal{H}$ consists of uniform quadratic elements of $K\Gamma$ and generates the ideal $I_\Gamma$. By inspection, $\CTip(h) = 1$ for all $h \in {\mathcal H}$ so condition (i) of Theorem~\ref{reducedG} holds.

(ii). Let $h, h'$ be distinct elements of $\mathcal{H}$. Since both $h$ and $h'$ have length 2, if $h$ reduces over $h'$ then $\Tip(h') \in \Supp(h)$. But, by inspection of $\mathcal{H}$ we see this is never the case. Thus $h$ does not reduce over $h'$ and condition (ii) holds.

(iii). We now consider the overlap difference for two (not necessarily distinct) members of
$\mathcal{H}$. By inspection, it is clear that in order to have an
overlap difference of elements $h, h' \in \mathcal{H}$, then at least one of $h, h'$ is a monomial.
Moreover the overlap difference of two monomial elements always reduces to zero. Thus we only need
to consider the overlap difference of one monomial element and one non-monomial element.
Label the non-monomial elements of $\mathcal{H}_2$ as $h_1 = \b_6 \gamma_3 - \gamma_8 \b_3 , h_2 = \b_6 \gamma_5 - \gamma_8 \b_5 , h_3 =  \b_6 \gamma_8 - \gamma_8 \b_8, h_4 =  \b_7 \gamma_1 - \gamma_6 \b_1, h_5 = \b_7 \gamma_6 - \gamma_6 \b_6, h_6 =  \b_8 \gamma_2 -\gamma_7 \b_2, h_7 = \b_8 \gamma_4 - \gamma_7 \b_4, h_8 = \b_8 \gamma_7 - \gamma_7 \b_7. $

We start with $h_1 = \b_6 \g_3 - \g_8 \b_3$ and $\a_{13} \b_6$.
The overlap difference $o(\a_{13} \b_6, h_1, \g_3, \a_{13})  = \a_{13} \b_6 \g_3 - \a_{13} \b_6 \g_3 + \a_{13} \g_8 \b_3 = \a_{13} \g_8 \b_3$.
Now $\a_{13} \g_8 \in {\mathcal H}$. So $o(\a_{13} \b_6, h_1, \g_3, \a_{13}) \Rightarrow_{\mathcal{H}} 0$.
Most cases follow this pattern of a simple reduction of the overlap difference over ${\mathcal H}$ to zero.
However, a few cases require two reductions. For example, consider $h_3 =  \b_6 \gamma_8 - \gamma_8 \b_8$ and $\g_8 \g_2$. We have $o(h_3, \g_8 \g_2, \g_2, \b_6)  =  \b_6 \g_8 \g_2 - \g_8 \b_8 \g_2 - \b_6 \g_8 \g_2 = - \g_8 \b_8 \g_2$. Now $h_6 = \b_8 \g_2 - \g_7 \b_2 \in {\mathcal H}$ so $o(h_3, \g_8 \g_2, \g_2, \b_6) \Rightarrow_{\mathcal H} -\g_8 \g_7 \b_2$. Next we observe that $\g_8 \g_7 \in {\mathcal H}$ so a second reduction gives that $o(h_3, \g_8 \g_2, \g_2, \b_6) \Rightarrow_{\mathcal H} 0$.

In this way, we can show that the overlap difference of any two elements $h, h' \in \mathcal{H}$ reduces to zero so (iii) holds. Therefore $\mathcal{H}$ satisfies the conditions of Theorem~\ref{reducedG}, and $\mathcal{H}$ is a reduced Gr\"{o}bner basis of the ideal $I_\Gamma$.
\end{proof}

The following result is now immediate using Theorem~\ref{Kosz}(ii).

\begin{thm}
Let $\L$ be the $(6,2)$-stacked algebra of Example~\ref{nonfish}. Let $\hat{E}(\L)$ be the Ext algebra of $\L$ with the hat-degree grading. Then $\hat{E}(\L)$ is a Koszul algebra.
\end{thm}

\end{example}

We end with some open questions. Firstly, is it the case that the regraded Ext algebra of any
$(D,A)$-stacked algebra with $D >2, A>1, D \neq 2A$ and $D \neq A+1$ is always a Koszul algebra? Secondly, it was shown in
\cite[Theorem 3.6]{bib:GS3} that, for monomial algebras of infinite global
dimension, the $(D, A)$-stacked monomial algebras are precisely
the monomial algebras for which every projective module $P^n$ in a minimal graded projective resolution
$(P^n, d^n)$ of $\Lambda_0$ is generated in a single degree and for which the Ext algebra $E(\L)$ is finitely generated. Can this characterisation be extended to all $(D, A)$-stacked algebras? More generally, what can be said about the class of finite-dimensional algebras $\L$ with the property that the $n$th projective module in a minimal projective resolution of $\L/\rrad$ is generated in a single degree?

\end{document}